\documentclass[12pt]{amsart}
\usepackage{amsthm,amsmath,latexsym,amssymb,amsmath,mathrsfs}
\usepackage{fancybox,ascmac}
\usepackage[dvipdfmx]{graphicx}

\usepackage{txfonts}
\usepackage{graphicx}
\usepackage{wrapfig}
\usepackage{pifont}
\usepackage{color}

\usepackage[left=30mm,right=30mm,top=30mm,bottom=30mm]{geometry}

\newtheorem{dfn}{Definition}[section]
\newtheorem{thm}[dfn]{Theorem}
\newtheorem{lem}[dfn]{Lemma}

\newtheorem{rem}[dfn]{Remark}

\newtheorem{prop}[dfn]{Proposition}\makeatletter

\makeatletter
\renewenvironment{proof}[1][\proofname]{\par
  \normalfont
  \topsep6\p@\@plus6\p@ \trivlist
  \item[\hskip\labelsep{\bfseries #1}\@addpunct{\bfseries.}]\ignorespaces
}{%
  \endtrivlist
}
\renewcommand{\proofname}{Proof.}
\makeatother

\numberwithin{equation}{section}

\begin{document}

\title{Geometric structures of late points of a two-dimensional simple random walk}

%    Remove any unused author tags.

%    author one information
\author{Izumi Okada}

%\runtitle{Geometric structures of late points}
%\thankstext{T1}{Footnote to the title with the ``thankstext'' command.}

%\begin{aug}

%\author{\fnms{Izumi}  \snm{Okada}%\corref{}\thanksref{t2}
%\ead[label=e1]{izumiokada1205@gmail.com}}

%\author{Izumi Okada}
%\thanksref{t1,t2,m1}
%\ead[label=e1]{first@somewhere.com}}

%\runauthor{Izumi Okada}

%\affiliation{Tokyo University of Science}

%\address{Department of Information Sciences\\
%Tokyo University of Science, 2641 Yamazaki, Noda, Chiba, 278-0022, Japan \\
%\printead{e1}\\
%\phantom{E-mail:\ izumiokada1205@gmai.com }
%}

%\address{Address of the Third author\\
%Usually a few lines long\\
%Usually a few lines long}
%\printead{e3}\\
%\printead{u1}}
%\end{aug}
\date{}

\dedicatory{}

\begin{abstract}
As Dembo ($2003$, $2006$) suggested, we consider the problem of late points for a simple random walk in two dimensions. 
It has been shown that the exponents for the number of pairs of late points coincide with those of favorite points and high points in the Gaussian free field, whose exact values are known.  
We determine the exponents for the number of $j$-tuples of late points on average. 
\end{abstract}

%\begin{keyword}[class=MSC]
%\kwd[Primary ]{60G60}
%\kwd{60K35}
%\kwd[; secondary ]{60J10}
%\end{keyword}

\maketitle

\section{Introduction }\label{intro}
This paper discusses the properties of special sites, called late points, in a two-dimensional random walk. 
The cover time is the time taken to randomly walk in $\mathbb{Z}^2_n(=\mathbb{Z}^2/n\mathbb{Z}^2)$ and visit every point of $\mathbb{Z}^2_n$, and 
a late point of a random walk in $\mathbb{Z}^2_n$ is a point of $\mathbb{Z}_n^2$, 
where the first hitting time is nearly equal to the cover time in a certain specific sense. 
We denote the set of $\alpha$-late points 
in $\mathbb{Z}^2_n$ as $\mathcal{L}_n(\alpha)$ for $0<\alpha<1$ as in \cite{Dembo2} (see (\ref{ff1+}) in the next section) 
and  obtain certain asymptotic forms of
\begin{align}\label{correla}
|\{ \vec{x}\in \mathcal{L}_n(\alpha)^j:
d(x_i,x_l)\le n^{\beta} \text{ for any }1\le i,l \le j  \}|
\end{align}
for any $0< \alpha,\beta <1$ and $j\in \mathbb{N}$, where $\vec{x}:=(x_1,\ldots,x_j)$. 
We then solve the related problem posed in Open Problem $4$ in \cite{Dembo*} and Open Problem $4.3$ in \cite{Dembo**}. 

Approximately $60$ years ago, Erd\H{o}s and Taylor \cite{er} proposed a problem concerning a simple random walk in
$\mathbb{Z}^d$. Forty years later, Dembo, Peres, Rosen, and Zeitouni \cite{Dembo1,Dembo3} solved it and other related problems by
developing innovative proofs. These methods yielded results concerning late points in $\mathbb{Z}_n^2$, verified by Dembo et al. 
 \cite{Dembo2}, which showed that the numbers of late points in clusters of different sizes have a variety of power growth exponents. 
 These methods and tools have now been improved. 
 Belius-Kistler \cite{bel} introduced a new multi-scale refinement of the $2$nd moment method. 
 In these estimates, it is more difficult to deal with the lower bound of some numbers than the upper one. 

Conversely, in this study, it is more difficult to compute the upper bound. 
We explain why the results in \cite{ho,Dembo2} cannot be easily extended to arbitrary $j$-tuples of points. 
In \cite{ho,Dembo2}, they estimated the probability that the pairs of points are late points. 
The number of pairs of $\alpha$-late points can be easily computed by this probability. 
However, because the probability is complex, the number of arbitrary $j$-tuples of late points cannot be easily computed 
by the probability that $j$-tuples of points are late points  
(see the explanation of the proofs for the main result in Section \ref{known}). 
Thus, we use a linear algebra approach 
by exploiting the relationship between the probability and ultrametric matrices. 
We find the relationship by estimating the probability with a Green's function. 
%Thus, we use the relationship between the probability and ultrametric matrices. 
%In Section \ref{known}, we explain how the linear algebra approach is related to the random walk cover time problem. 

Here, we explain the motivation for studying $\alpha$-late points. 
%First, we want to estimate the correlation between the local time of unbounded points. 
%As a first step, we estimate the finite correlation in this study. 
We want to compare the asymptotic behavior of special points in a random walk and in the Gaussian free field (GFF) 
by understanding their similarity between the local time and the GFF. 
In fact, there are several known results concerning similarity. 
Eisenbaum et al. \cite{Eisen} showed a powerful equivalence law called the generalized second Ray-Knight theorem for a random walk and the GFF. 
Ding et al. \cite{Ding1,Ding2} showed a strong connection between the expected maximum of the GFF and the
expected cover time. In addition, for $0<\alpha <1$, they used the set of $\alpha$-high points in the GFF in $\mathbb{Z}_n^2$
(sites where the GFF takes high values) and $\alpha$-favorite points in $\mathbb{Z}^2$ 
(sites where the local time is close to that of the most frequently visited site). 
Dembo et al.  \cite{Dembo2} and Brummelhuis and Hilhorst  \cite{ho} estimated the number of pairs of 
$\alpha$-late points, and Daviaud \cite{ol}  estimated the $\alpha$-high points. 
We show the corresponding results for the $\alpha$-favorite points in our forthcoming paper. 
The similarity between $\alpha$-late points, high points, and favorite points are included in these estimations. 
In addition, we find that local times converge the GFF in long-time through the generalized second Ray-Knight theorem. 
We expect that the similarity helps us to understand the convergence.

%We will estimate that for $\alpha$-favorite points  in a forthcoming paper. 
%In each case, the exponent coincides with  that of  (A) or (B). 

%Second, we observe a large deviation estimate of the local time, 
%which is the estimation of the probability that the local time is much lower than the average. 
%Note that  the first moment of (\ref{correla}) is equal to
%\begin{align*}
%\sum_{\substack{x_1,...,x_j \in \mathbb{Z}_n^2, \\
%d(x_i,x_l)\le n^{\beta}, 1\le  \forall i,l \le j  }}P(x_1,...,x_j  \in \mathcal{L}_n(\alpha)).
%\end{align*}
%$P(x_1,...,x_j  \in \mathcal{L}_n(\alpha))$ corresponds to the probability that the local times of $x_1,...,x_j $ are zero. 
%Then, the key of estimating the first moment of (\ref{correla}) is to compute the large deviation estimate of the local times. 
%This is similar to the research of $\alpha$-high points and $\alpha$-favorite points. 
%In general, the limit theorem for random variables has a connection with its large deviation estimates. %or high moment. 
%We believe that these estimations of the local time carry the asymptotic of various functionals of the local time. 
%In fact, as Proposition \ref{prop9*} mentions later, these are connected with the geometric structures of $\alpha$-late points. 
%In addition, result (A) is almost equivalent to estimates of the second moments of  $|\mathcal{L}_n(\alpha)|$. 

%%%%%%%%%%%%%%%%%%%%%%%%%%%%%%%%%%%%%%%%%%%%%%%%%%%%%%%%%%%%%%%%%%%%%%%%%%%%%%%%%%%%%%%%%%%%%%%%%%%%%%%%%%%%%%%%%%%%%%%%%%%%%%%%
\section{Known results and main results}\label{known}
To state our main results, we introduce the following notation. 
Let $d$ be the Euclidean distance and $\mathbb{N}:=\{1,2,\cdots\}$. 
For $n\in \mathbb{N}$, let $D(x,r):=\{y\in \mathbb{Z}^2_n: d(x,y)< r\}$ and for any $G\subset \mathbb{Z}^2_n$, 
$\partial G:=\{y\in G^c: d(x,y)=1 \text{ for some }x\in G \}$. 
For $x\in \mathbb{Z}^2_n$, we sometimes omit $\{\}$ when writing the one-point set $\{x\}$. 
Let $\{S_k\}_{k=1}^{\infty}$ be a simple random walk in $\mathbb{Z}^2_n$. 
Let $P^x$ denote the probability of a simple random walk starting at $x$. 
For simplicity, we write $P$ for $P^0$.  
Let $K(n,x)$ be the number of times visits for the simple random walk to $x$ up to time $n$, 
that is, $K(n,x)=\sum_{i=0}^n1_{\{S_i=x\}}$. 
For any $D\subset \mathbb{Z}^2_n$, let $T_D:=\inf \{m\ge1: S_m\in D\}$. 
%For simplicity, we write $T_X$ for $T_{\{x_1,...,x_j\}}$. 
Let $\tau_n:=\inf \{m\ge 0: S_m\in \partial D(0,n)\}$. 
$\lceil a \rceil$ denotes the smallest integer $n$ with $n \ge a$. 
We use the same notation for a simple random walk in $\mathbb{Z}^2$.

We introduce the known results for $\alpha$-late points in $\mathbb{Z}^2_n$. 
Dembo et al. \cite{Dembo3} estimated the asymptotic form of the cover time of a simple random walk in $\mathbb{Z}_n^2$ as follows:
\begin{align*}
\lim_{n \to \infty} \frac{\max_{x\in {\mathbb Z}^2_n}T_x}{(n\log n)^2}= \frac{4}{\pi} \quad \text{ in probability.}
\end{align*}
For $0<\alpha<1$, we define the set of $\alpha$-late points in $\mathbb{Z}^2_n$ such that
\begin{align}\label{ff1+}
\mathcal{L}_n(\alpha):=\bigg\{ x\in {\mathbb{Z}^2_n} :
\frac{T_x}{(n\log n)^2}\ge \frac{4\alpha}{\pi} \bigg\}.
\end{align}
Brummelhuis and Hilhorst \cite{ho} estimated the average of (\ref{correla}) for $j=2$, and Dembo et al. \cite{Dembo2} estimated (\ref{correla}) in probability for $j=2$. 
%Theorem $1.4$ of \cite{Dembo2} estimated the power growth exponent of pairs of $\alpha$-late points within distance $n^{\beta}$ of each other.  
We extend this to a \emph{full multi-fractal analysis}. 
%%%%%%%%%%%%%%%%%%%%%%%%%%%%%%%%%%%%%%%%%%%%%%%%%%%%%%%%%%%%%%%%%%%%%%%%%%%%%%%%%%%%%%%%%%%%%%%%%%

\begin{thm}\label{h2}
For any $0< \alpha,\beta <1$ and $j\in \mathbb{N}$
\begin{align*}
&\lim_{n\to \infty}\frac{\log 
E[|\{ \vec{x}\in \mathcal{L}_n(\alpha)^j:
d(x_i,x_l)\le n^{\beta} \text{ for any }1\le i,l \le j  \}|]}{\log n}\\
&=\hat{\rho}_j(\alpha, \beta),
\end{align*}
where
\begin{align*}
\hat{\rho}_j(\alpha, \beta):=
\begin{cases}
 2+2(j-1)\beta-\frac{2j\alpha}{(1-\beta)(j-1)+1}&(\beta\le 1+\frac{1-\sqrt{j\alpha}}{j-1}),
\\
2(j+1-2\sqrt{j\alpha})&(\beta\ge 1+\frac{1-\sqrt{j\alpha}}{j-1}).
\end{cases}
\end{align*}
\end{thm} 
% We will extend the condition $d(x_i,x_l)\le n^{\beta} $  to be more general in Section \ref{cor++}. 

\begin{rem}\label{h4}
We are preparing a paper on the following result: 
for any $0< \alpha,\beta <1$, and $j\in \mathbb{N}$ in probability 
\begin{align*}
&\lim_{n\to \infty}\frac{\log 
|\{ \vec{x}\in \mathcal{L}_n(\alpha)^j :
d(x_i,x_l)\le n^{\beta} \text{ for any }1\le i,l \le j \}|}{\log n}\\
&=\rho_j(\alpha, \beta),
\end{align*}
where
\begin{align*}
\rho_j(\alpha, \beta):=
\begin{cases}
 2+2(j-1)\beta-\frac{2j\alpha }{(1-\beta)(j-1)+1}&(\beta\le \frac{j}{j-1}(1-\sqrt{\alpha})),
\\
4j(1-\sqrt{\alpha})-2j(1-\sqrt{\alpha})^2/\beta&(\beta\ge \frac{j}{j-1}(1-\sqrt{\alpha})).
\end{cases}
\end{align*}
An explanation of the difference in the exponents is given in \cite{Dembo*,Dembo**} for $j=2$. 
%The key lemmas of this estimate and Theorem \ref{h2} are same and 
%our proofs of this estimate and Theorem \ref{h2} need a liner algebra approach. 
%After we obtain Theorem \ref{h2}, long computations, which can be derived from \cite{Dembo2}, only remain to prove it.    
%Then, we omit the proof in this paper since we want to attend a liner algebra approach. 
\end{rem}

Now, we provide an explanation of the proofs for the main result. 
In particular, we explain how this problem is connected to the linear algebra approach. 
Roughly speaking, 
certain asymptotic forms of (\ref{correla}) are determined using the hitting probabilities of $j$-points of a simple random walk. 
In addition, the hitting probabilities are determined by Green's functions of $j$-points, 
and the values of Green's functions of $j$-points behave with ultrametricity  in long-time. 
Proposition \ref{kii} yields that we can reduce the configurations of $j$-points to those in an ultrametric position. 
That is why ultrametricity plays an important role in the main result. 

Now, we provide the details. 
For the proof of Theorem \ref{h2}, we must find an appropriate estimate of 
\begin{align}\notag
&E[|\{ \vec{x}\in \mathcal{L}_n(\alpha)^j :
d(x_i,x_l)\le n^{\beta}  \text{ for any }1\le i,l \le j \}|]\\
\label{dec}
=&\sum_{\substack{d(x_i,x_l)\le n^{\beta}, \\x_i\in \mathbb{Z}_n^2,  1\le \forall i,l \le j } }
P( \vec{x} \in \mathcal{L}_n(\alpha)^j).
\end{align}
%as we mentioned in Section \ref{intro}. 
Note that the position of a $j$-tuple point determines 
the value of $P( \vec{x} \in \mathcal{L}_n(\alpha)^j)$. 
%Then, we have to estimate $P( \vec{x} \in \mathcal{L}_n(\alpha)^j)$. 
%As we mentioned, we need to compute $P( \vec{x} \in \mathcal{L}_n(\alpha)^j)$. 
This value can be expressed by a matrix constructed 
from $G_n(x,y):=\sum_{m=0}^\infty P^x(S_m=y,m<\tau_n)$ for $x$, $y\in D(0,n)$, 
which is the Green's function of the walk killed when it exits $D(0,n)$.  
We shall show that to achieve uniformity in $x_1,\ldots,x_j \in D(0,n/3)$,
\begin{align}\label{app1}
&P( \vec{x} \in \mathcal{L}_n(\alpha)^j)
\approx \exp\bigg(-2\alpha \log n 
\chi \bigg(\bigg(\frac{\pi G_n(x_i,x_l)}{2\log n}\bigg)_{1\le i,l \le j}\bigg) \bigg),
\end{align}
where $a_n \approx b_n$ means $\log a_n/ \log b_n\to 1$ as $n\to \infty$ for any sequence and 
$\chi(A)$  is the summation over all the elements of  $A^{-1}$ for any regular matrix $A$. 
We explain the proof of (\ref{app1}) in step (I). 

(I) The proof of (\ref{app1})

%To show (\ref{app1}), we firstly estimate the abstract enumerate of  Markov chain's path in Section \ref{local+}. 
%n fact, we estimate the probability that  local times of some edges in the graph are specified. 
%We show that it can be written by the product of some transition probabilities of a Markov chain. 
%by short time, that is, 
%the information of the combination of paths conditioned by the excursions of Markov chain. 
In Section \ref{Theoh2} (Proposition \ref{prop8*}), we shall see the probability that $x_1,..., x_j$ in $D(0,n)$  will be  uncovered by the walk under a certain condition determined by the crossing number between two large circles 
%(see an example for $j=5$ in figure $1$) 
is used to estimate the left-hand side of (\ref{app1}). 
%\begin{figure}[htbp]
%\begin{center}
%\includegraphics[width=6cm,height=6cm,clip]{cro1.eps}
 %\caption{}
%\end{center}
%\end{figure}  
%We shall develop an argument concerning a general Markov chain in Section \ref{local+} and apply it to express the probability by means of a certain product of hitting probabilities in a subset of $\mathbb{Z}^2_n$ for the simple random walk. 
%It is known that element quantities of a simple random walk is related to all the elements of inverse matrix of potential kernel (See Chapter III in \cite{spi}).  
%We shall then estimate the hitting distribution until a certain stopping time in Section \ref{hit+}. 
In Section \ref{basic}, we obtain %simultaneous 
equations consisting of hitting probabilities and Green's functions (see (\ref{simeq})), 
which show that hitting probabilities can be expressed by 
certain cofactors of $(G_n(x_i,x_l))_{1\le i,l \le j}$ (see (\ref{hha})).  
Finally, we find that the product is equal to the right-hand side of (\ref{app1}). 
%Our results imply that there exists a strong connection between such hitting probabilities and matrices of a Green's function of a simple random walk and, hence, we obtain (\ref{app1}). 
%To obtain Lemma \ref{d1*}, we need to compute hitting probabilities to a variety of subsets and stopping times, 
%which can't be simply computed. 

Next, we provide an explanation of the proof of Theorem \ref{h2} assuming (\ref{app1}). 
We explain the difficulty of the proof of the upper bound. 
In fact,  by using (\ref{app1}), we find that the logarithm of (\ref{dec}) is asymptotically equal to that of the summation of 
\begin{align}\label{app1*}
\exp\bigg(-2\alpha \log n\chi \bigg(\bigg(\frac{\pi G_n(x^{(n)}_i,x^{(n)}_l)}{2\log n}\bigg)_{1\le i,l \le j} \bigg)\bigg)
\end{align}
over  $(x_1^{(n)},\ldots,x_j^{(n)}) \in D(0,n/3)^j$, where 
$(x_1^{(n)},\ldots,x_j^{(n)})$ is an ultrametric space  with the error term $n^{o(1)}$ as $n\to \infty$. 
%As we will mention in Remark \ref{gree1}, 
%$(\pi G_n(x^{(n)}_i,x^{(n)}_i)/(2\log n))_{1\le i,l \le j}$ is asymptotically close to the ultrametric matrix as $n\to \infty$ in a certain sense for such sequences $(x_1^{(n)},...,x_j^{(n)}) \in {(\mathbb{Z}^2_n})^j$. 
%For example, consider sequences $(x_1^{(n)},x_2^{(n)},x_3^{(n)})_{n\in \mathbb{N}} \in {(\mathbb{Z}^2_n})^3$. 
%Then, if we estimate  (\ref{dec}) for $j=3$, 
%by the symmetry of $x_1^{(n)}$, $x_2^{(n)}$ and $x_3^{(n)}$ 
%it is plausible that we need to consider the following sequence: 
%it is the set  the function of the distance of  a pair in $x_1^{(n)},...,x_j^{(n)}$ 
%corresponds to the ultrametric space . 
Here, the ultrametric space with the error term $n^{o(1)}$ is the following set with an associated distance function:  for any $1\le i,l,p\le j$ with 
$i\neq l $, $ l \neq p$, and $i\neq p$
\begin{align}\label{pos}
d(x_i^{(n)},x_l^{(n)})\le  \max \{ d(x_l^{(n)},x_p^{(n)})n^{o(1)}, d(x_p^{(n)},x_i^{(n)})n^{o(1)} \}. 
\end{align} 
%特定のconfigurationの間で，$n$に対するオーダーを2点間の関数として用いると，それが超距離になっている
Then, the configuration of $x_1^{(n)},...,x_j^{(n)}$ has a certain nesting structure. 
%That is why we can reduce configurations of $j$-points to those in an urtrametric position. 
 %(see an example for $j=5$ in figure $2$). 
%\begin{figure}[htbp]
%\begin{center}
%\includegraphics[width=10cm,height=7cm,clip]{ul1.eps}
% \caption{}
%\end{center}
%\end{figure}  
%where $a_n \approx b_n$ means $\log a_n/ \log b_n\to 1$ as $n\to \infty$ for any sequence. 
For example, if we estimate the upper bound of (\ref{dec}) for $j=3$, 
we need to look at an equidistant configuration and the position that the one is far from the others.  
For $j=3$, an equidistant configuration means a triple $(x_1^{(n)}, x_2^{(n)}, x_3^{(n)})$ such that as $n\to \infty$,
\begin{align*}
d(x_1^{(n)}, x_2^{(n)})
\approx d(x_2^{(n)}, x_3^{(n)})
\approx d(x_3^{(n)}, x_1^{(n)}).
\end{align*}
%Then, we need to estimate the number of points and the value of of $P( x_1,x_2,x_3\in \mathcal{L}_n(\alpha))$ for each case of (A) and (B). 
For a general $j\in \mathbb{N}$, there are various positions of $x_1,...,x_j$. 
%We need to consider any $(x_1^{(n)},...,x_j^{(n)} )_{n\in \mathbb{N}}$ which satisfies (\ref{pos}). 
%Roughly speaking, we need to consider any $\{x_i^{(n)}\}_{1\le i \le j}$ such that for any $1\le i, l,p \le j$ with  $i \neq l, l\neq p,p\neq i$,  
%\begin{align}\notag
%&\lim_{n\to \infty} \log x_i^{(n)}/ \log x_l^{(n)}:=b_{i,l} \text{ exists}\\
%\label{pos4}
%\text{and }\quad \quad\quad&b_{i,l}> b_{i,p} \Rightarrow b_{l,p}=b_{i,l}.
%\end{align} 
Therefore, when $j$ increases, computing the upper bound in Theorem \ref{h2} becomes difficult. 
%Therefore, it is difficult to compute for general $j\in \mathbb{N}$. 
Subsequently, we developed the following unique step. 
% which has original and important conceptions to solve only the upper bound. 
%Now, we provide the details of the step (II). 

(II) The upper bound in Theorem \ref{h2} by assuming (\ref{app1})

We need to find the leading term of (\ref{app1*}) over 
$(x_1^{(n)},\ldots,x_j^{(n)})$ conditioned by (\ref{pos}). 
We will show that 
$(\pi G_n(x^{(n)}_i,x^{(n)}_l)/(2\log n))_{1\le i,l \le j}$ is asymptotically close to the ultrametric matrix as $n\to \infty$ in a certain sense. 
Therefore, we define $\mathcal{M}_j$  (see Section \ref{matrix}), which is a certain set of $j\times j$-ultrametric matrices (see Section $3.3$ in \cite{ale}), and estimate $\chi(A)$ for any $A$ in $\mathcal{M}_j$ to further estimate $\chi((\pi G_n(x^{(n)}_i,x^{(n)}_l)/ (2\log n))_{1\le i,l \le j})$. 
%To show it, we need to observe properties of $\mathcal{M}_j$ for $j\in \mathbb{N}$, which correspond to properties of a variety of positions. 
%Now, we introduce properties of $\mathcal{M}_j$. 
Ultrametric matrices have come to the attention of some linear algebraists and 
have been used as models of systems that can be represented by a bifurcating hierarchical tree (e.g., see \cite{mark}). 
In this study, we find new properties of $\mathcal{M}_j$. 
Proposition \ref{prop5*} 
%is  important since the proof  depends on remarkable properties of $\mathcal{M}_j$ 
%and it corresponds to a key estimate for main results. 
yields the minimum of $\chi(A)$ for $A$ in $\mathcal{M}_j$ under a certain condition. 
Finally, we obtain the result that the properties of $\mathcal{M}_j$ directly determine the asymptotic behavior of  (\ref{app1*}) 
and that the leading term comes from the equidistant configuration for any $j \in \mathbb{N}$. 
\section{Basic properties}\label{basic}
In this section, we use the preliminary results concerning a simple random walk that will be applied in later sections. 
In proofs given in the remainder of this paper, we use constants that may vary for different occurrences. 

%%%%%%%%%%%%%%%%%%%%%%%%%%%%%%%%%%%%%%%%%%%%%%%%%%%%%%%%%%%%%%%%%%%%%%%%%%%%%%

\subsection{Hitting probabilities}\label{hit+}
First, we compute the probabilities that a simple random walk in 
$\mathbb{Z}^2_n$ does not hit a $j$-tuple point until a certain random time.  
Given the $j$ distinct points $x_1,\ldots,x_j$ of $\mathbb{Z}^2_n$ and 
a non-empty subset $\tilde{D}$ of $\mathbb{Z}^2_n$ that is disjoint from $X:=\{x_1,\ldots,x_j\}$,  
let $\tilde{\tau}$ denote a time when the walk enters $\tilde{D}$. 
For $1\le i,l\le j$, and $y \not\in X$, we define 
\begin{align*}
q_{i,l}&:=\sum_{m=0}^\infty P^{x_i}(S_m=x_l ,m<\tilde{\tau}\wedge T_y ),\\
Q&:=(q_{i,l})_{1\le i,l \le j}.
\end{align*}
\begin{lem}\label{qe2}
For $1\le u\le j$, it holds that
\begin{align}
\label{hha}
P^{x_u}(T_y=T_X\wedge \tilde{\tau}\wedge T_y)=
\sum_{i=1}^j
\frac{\text{(cofactor of }q_{u,i})}{\mathrm{det}(Q)}
P^{x_i}(T_y=\tilde{\tau}\wedge T_y). 
\end{align}
We have
\begin{align}
\notag
\min_{1\le u \le j}P^{x_u}(T_y=\tilde{\tau}\wedge T_y)\chi(Q)
\le &\sum_{u=1}^j P^{x_u}(T_y=T_X\wedge \tilde{\tau}\wedge T_y)\\
\label{hha*}
\le &\max_{1\le u\le j}P^{x_u}(T_y=\tilde{\tau}\wedge T_y)\chi(Q). 
\end{align}
Note that for any regular matrix $A$, 
$\chi(A)$ is the summation over all the elements of $A^{-1}$. 
\end{lem}
\begin{proof}
Because the summation of both sides of (\ref{hha}) over $1\le u\le j$ yields (\ref{hha*}), 
it suffices to show (\ref{hha}). 
By decomposing the probability $P^{x_i}(T_y= \tilde{\tau}\wedge T_y)$ 
according to  the last time the walk leaves the set $X$ 
before $\tilde{\tau}\wedge T_y$, we obtain
\begin{align}\label{simeq}
P^{x_i}(T_y= \tilde{\tau}\wedge T_y)=
\sum_{l=1}^j q_{i,l}P^{x_l}(T_y=T_X\wedge \tilde{\tau}\wedge T_y),
\end{align}
for $1\le i\le j$. 
The matrix $Q$ is regarded as the Green kernel for the Markov chain on $X$ with the substochastic transition matrix 
$U:=(u_{i,l})_{1\le i,l \le j}$ given by 
\begin{align*}
u_{i,l}:=P^{x_i}(T_{x_l}=T_X<\tilde{\tau} \wedge T_y), 
\end{align*}
so that $UQ =Q-E$, where $E$ denotes the unit matrix. Accordingly, $Q$ is regular and
\begin{align}\label{hhu**}
Q^{-1}
:=E-U.
\end{align}
Therefore, we have (\ref{hha}). 
\end{proof}

Next, we introduce the estimates of the hitting probabilities for a simple random walk in $\mathbb{Z}^2$, as we only need estimates for ``$\mathbb{Z}^2$'' in this paper. 
\begin{lem}\label{qe3}
To achieve uniformity in $0<r<|x|<R$,
\begin{align}\notag
P^x( T_0<\tau_R)&=\frac{\log (R/|x|)+O(|x|^{-1})}{\log R}(1+O((\log |x|)^{-1})),\\
\label{g2}
P^x( \tau_r<\tau_R)&=\frac{\log (R/|x|)+O(r^{-1})}{\log(R/{r})}. 
\end{align}
\end{lem}
\begin{proof}
As per Exercise $1.6.8$ in \cite{Law}, or ($4.1$) and ($4.3$) in \cite{rosen}, we obtain the desired result. 
\end{proof}

Next, we give the estimates of a Green's function. 
For $x$, $y\in D(0,n)$, $G_n(x,y)$ is a Green's function. 
\begin{lem}\label{qe4}
For any $x\in D(0,n)$
\begin{align}\notag
G_n(x,0)=&\sum_{m=0}^\infty P^x(S_m=0,m<\tau_n)\\
\label{gg2**}
=&\frac{2}{\pi}\log \bigg(\frac{n}{d(0,x)^+}\bigg)+O((d(0,x)^+)^{-1}+n^{-1}+1),
\end{align}
where $\tau_n$ is the stopping time as we mentioned in Section \ref{known} and 
$a^+=a\vee 1$. 
In particular, for $x,y \in D(0,n/3)$,
\begin{align}\label{gg2*}
G_n(x,y)
=\frac{2}{\pi}\log \bigg(\frac{n}{d(x,y)^+}\bigg)+O((d(x,y)^+)^{-1}+n^{-1}+1).
\end{align}
\end{lem}
\begin{proof}
As per Proposition $1.6.7$ in \cite{Law} or ($2.1$) in \cite{rosen}, we obtain (\ref{gg2**}). 
Therefore, for $x,y \in D(0,n/3)$,
\begin{align*}
G_n(x,y)&\le\sum_{m=0}^\infty P^{x-y}(S_m=0,m<\tau_{4n/3})\\
&=\frac{2}{\pi}\log \bigg(\frac{n}{d(x,y)^+}\bigg)+O((d(x,y)^+)^{-1}+n^{-1}+1),\\
G_n(x,y)&\ge \sum_{m=0}^\infty P^{x-y}(S_m=0,m<\tau_{2n/3})\\
&=\frac{2}{\pi}\log \bigg(\frac{n}{d(x,y)^+}\bigg)+O((d(x,y)^+)^{-1}+n^{-1}+1).
\end{align*}
Subsequently, we obtain (\ref{gg2*}). 
\end{proof}
\begin{rem}
In addition, with the aid of (\ref{gg2**}), the strong Markov property yields 
\begin{align}\label{g2*}
P(\tau_n< T_0)
=(\sum_{m=0}^\infty P(S_m=0,m<\tau_n))^{-1}
=\frac{\pi}{2\log n}(1+o(1)).
\end{align}
\end{rem}
%%%%%%%%%%%%%%%%%%%%%%%%%%%%%%%%%%%%%%%%%%%%%%%%%%%%%%%%%%%%%%%%%%%%
%%%%%%%%%%%%%%%%%%%%%%%%%%%%%%%%%%%%%%%%%%%%%%%%%%%%%%%%%%%%%%%%%%%%%%%%%%%%%%%%%%%%%%%%%%%%%%%%%%
\section{Proof of Theorem \ref{h2}}\label{Theoh2}
In this section, we provide the proof of Theorem \ref{h2}. 
We give estimates for the proof of Theorem \ref{h2} in Section \ref{some in Theoh2} and 
 the proof of Theorem \ref{h2} in Section \ref{Theoh2*}. 
\subsection{Some estimates for the proof of Theorem \ref{h2}}\label{some in Theoh2}
To prepare estimates for the main result, we add the following definitions. 
We fix  $j\in \mathbb{N}$. 
%Let $\tilde\mathcal{M}_j=\tilde\mathcal{M}_j(\eta)$ be the set of all $j\times j$ matrices $(a_{i,l})_{1\le i,l \le j}$ satisfying the following properties: 
%if the metric space $(\{ \tilde{x}_1,...,\tilde{x}_j\}, \tilde{d})$ satisfies $\tilde{d}(\tilde{x}_i,\tilde{x}_l)=a_{i,l}$ for $1\le i,l \le j$, $(\{ \tilde{x}_1,...,\tilde{x}_j\}, \tilde{d})$ is the ultrametric space and $a_{i,l}\ge \eta$ for any $1\le i \neq l \le j$, that is, 
%\begin{enumerate}
%\item[($\tilde{a}$)]  symmetric,
%\item[($\tilde{b}$)]    $1\le  \forall i \neq l \le j$, $a_{i,i}=0$, $\eta \le a_{i,l}\le 1$, 
%\item[($\tilde{c}$)]   $1\le  \forall i, l, p  \le j$  with $i \neq l, l\neq p,p\neq i$, $a_{i,l} \le \max \{ a_{l,p}, a_{i,p}\}$.
%\end{enumerate}
%In addition, we let $\mathcal{M}_j:= \{(a_{i,l})_{1\le i,l \le j} :  (1-a_{i,l})_{1\le i,l \le j} \in \tilde\mathcal{M}_j \}$. 
For $0< \eta \le (1-\beta) \wedge \beta$, 
let $\mathcal{M}_j=\mathcal{M}_j^{\beta, \eta}$ be the set of $j\times j$-matrices $(a_{i,l})_{1\le i,l \le j}$ 
satisfying the following properties: 
\begin{enumerate}
\item[$(a)$]  symmetric,
\item[$(b)$]   $a_{i,i}=1$, $1-\beta \le a_{i,l}\le 1-\eta$ for any $1\le i \neq l \le j$, and 
\item[$(c)$]   $a_{i,l} \ge \min \{ a_{l,p}, a_{i,p}\}$ for any 
$1\le i, l, p  \le j$  with $i \neq l, l\neq p,p\neq i$. 
\end{enumerate}
A strictly ultrametric matrix is a symmetric matrix with nonnegative entries that satisfies $(c)$; in addition, 
$a_{i,i}> \max\{a_{i,k}: k \in \{1,\ldots,i-1,i+1,\ldots,j\}\}$ for any $1\le i\le j$ (see \cite{mar}). 
Subsequently, any element in $\mathcal{M}_j$ is an ultrametric matrix. 
In Proposition \ref{prop2*}, 
we will show any matrix in  $\mathcal{M}_j$ is a regular matrix. 
%We fix $0<\beta<1$. 
%To introduce the lemmas, for any $0<\delta, \eta<1$ with $\beta-100\eta>0$ and $N$ with $\delta N=\beta-2\eta$,  
%let $s_1:=\eta$, $s_i:=s_{i-1}+\delta$ for any $i\ge2$ and $\mathcal{P}:=\{s_i\}_{i=1}^{N-1}$. 
%For $0< \eta \le (1-\beta) \wedge \beta$, let
%\begin{align*}
  %\mathcal{M}_j^{\beta, \eta}:=\mathcal{M}_j\cap  
  %\{(a_{i,l})_{1\le i,l \le j}:  1-\beta \le a_{i,l}\le 1-\eta \text{ for any }1\le i\neq l \le j \}.
%\end{align*}
Given the real-valued $j\times j$-matrices $M:=(m_{i,l})_{1\le i,l \le j} $ and 
$M':=(m'_{i,l})_{1\le i,l \le j}$, let 
\begin{align*}
 \mathcal{E}[M, M']
=& \mathcal{E}[M, M'](j,n)\\
:=&\{ \vec{x}\in (\mathbb{Z}_n^2)^j :m_{i,l} \le d(x_i,x_l)\le m'_{i,l} \text{ for any }1\le i\neq l \le j \}.
\end{align*}
Note that the set is independent of the diagonal elements of a matrix. 
When $m_{i,l}=a$ and $m'_{i,l}=a'$ for $1\le i\neq l \le j$, we simply write $\mathcal{E}[(a), (a')]$. 
For $A \in   \mathcal{M}_j$ and $\delta>0$ let
\begin{align*}
\hat{\mathcal{E}}_{\delta}[A]
=\hat{\mathcal{E}}_{\delta}[A](j,n)
:=&\mathcal{E}\bigg[\bigg(\frac{1}{2^j}n^{1-a_{i,l}}\bigg)_{1\le i,l \le j}, 
(2^j n^{1-a_{i,l}+\delta})_{1\le i,l \le j}\bigg].  
\end{align*}
By the following proposition, we find that 
it is possible to reduce the configuration of points to those in an ultrametric position. 

\begin{prop}\label{kii}
Fix $0<\beta<1$. 
For any $0<\delta<1-\beta$, 
there exists $n_0 \in \mathbb{N}$ such that for any $n\ge n_0$ and $\vec{x}\in \mathcal{E}[ (n^{\eta}) , (n^\beta) ]$, 
there exists $A\in \mathcal{M}_j^{\beta,\eta} $ such that  $\vec{x}\in \hat{\mathcal{E}}_{\delta}[A]$ holds. 
\end{prop}

We will show Proposition \ref{kii} in Section \ref{various pro}. 
Next, we introduce our goal in this subsection.  

\begin{prop}\label{prop8*}
%Fix $0<\alpha, \beta <1$ and $j \in \mathbb{N}$. 
Fix $0<\beta<1$. 
For any $\epsilon>0$, there exist $C>0$  and $0<\delta<1-\beta$ 
such that for any $0<\alpha <1$, $x\in \mathbb{Z}_n^2$ and all sufficiently large $n\in \mathbb{N}$ that satisfy 
$A\in \mathcal{M}_j^{\beta,\eta}$ and 
$\vec{x} \in  \hat{\mathcal{E}}_{\delta}[A]$, 
it holds that 
\begin{align*}
P(\vec{x} \in \mathcal{L}_n(\alpha)^j)
\le 
Cn^{-2\alpha \chi(A)+\epsilon}
\end{align*}
and for any $0<\alpha <1$, $x\in \mathbb{Z}_n^2$ with $x_1 \in D(0,n/10)^c$
 and all sufficiently large $n\in \mathbb{N}$ that satisfy 
$A\in \mathcal{M}_j^{\beta,\eta}$ and 
$\vec{x} \in  \hat{\mathcal{E}}_{\delta}[A]$, 
it holds that 
\begin{align*}
 C^{-1} n^{-2\alpha \chi(A)-\epsilon}
\le P(\vec{x} \in \mathcal{L}_n(\alpha)^j).
\end{align*}
\end{prop}

To show the proposition, we prepare notations and provide the lemma.  
For $k,n\in \mathbb{N}$ with $k\le n$, 
let $n_n=n_n(\alpha):=\lceil 2\alpha n^2\log n \rceil$, 
$r_k:=k!$, and $K_n:=\lceil n^b r_n \rceil$ for $b \in [1,3]$. 
Let $\mathcal{R}^{x_1}_n=\mathcal{R}^{x_1}_n(\alpha)$ 
be the time until completion of the 
first $n_n(\alpha)$ excursions of the path from $\partial D(x_1, r_{n-1})$ to $\partial D(x_1,r_n)$ 
(see the definition of excursions in Lemma $2.3$  in \cite{Dembo2} et al). 
%We define the domain $\mathbf{Go}(z)$ for any $z \in \mathbb{Z}_{K_n}^2$ such that 
%\begin{align*}
%\mathbf{Go}(z):=\{ \vec{x} \in (\mathbb{Z}_{K_n}^2)^j:  x_1,...,x_j \in D(z, r_{n-2}) \}.
%\end{align*}
\begin{lem}
Fix $0<\beta<1$. 
For any $\epsilon>0$, there exist $C>0$  and $0<\delta<1-\beta$ 
such that for any $0<\alpha <1$, $x\in \mathbb{Z}_{K_n}^2$, and all sufficiently large $n\in \mathbb{N}$ that satisfy 
$A\in \mathcal{M}_j^{\beta,\eta}$ and 
$\vec{x} \in  \hat{\mathcal{E}}_{\delta}[A](j,K_n)$, 
it holds that 
\begin{align*}
P(T_X >\mathcal{R}^{x_1}_n(\alpha))
\le 
CK_n^{-2\alpha \chi(A)+\epsilon/2}.
\end{align*}
\end{lem}
\begin{proof}
By the strong Markov property, it suffices to show that 
uniformity in $y_0\in \partial D(x_1,r_{n-1})$ and 
$\vec{x}\in \hat{\mathcal{E}}_{\delta}[A]$, 
\begin{align}\label{res}
P^{y_0}(T_X<T_{\partial D(x_1, r_n)} )
=\frac{1 +o(1)}{n}\chi(A).
\end{align}
Note that for $\vec{x}\in \hat{\mathcal{E}}_{\delta}[A]$, 
$1\le i,l \le j$,
and $y_0\in \partial D(x_1,r_{n-1})$
\begin{align*}
\sum_{m=0}^{\infty}P^{x_i}(S_m=x_l,m<T_{\partial D(x_1,r_n)} \wedge T_{y_0} )
\le &G_{r_n}(x_i-x_1,x_l-x_1),\\
\sum_{m=0}^{\infty}P^{x_i}(S_m=x_l,m<T_{\partial D(x_1,r_n)} \wedge T_{y_0} )
\ge& G_{r_{n-1}}(x_i-x_1,x_l-x_1).
\end{align*}
Because $r_{n-1}/(2^jK_n^{\beta+\delta}) \ge 3$ for all sufficiently large $n\in \mathbb{N}$ 
and $x_i, x_l \in D(x_1, 2^jK_n^{\beta+\delta})$ hold,  
(\ref{gg2*}) yields
\begin{align*}
\sum_{m=0}^{\infty}P^{x_i}(S_m=x_l,m<T_{\partial D(x_1,r_n)} \wedge T_{y_0} )
=\frac{2}{\pi }(\log r_{n} - \log d(x_i,x_l)^++O(1)).
\end{align*}
Subsequently, (\ref{g2}) yields that to achieve uniformity in 
$\vec{x}\in \hat{\mathcal{E}}_{\delta}[A]$  and $y_0\in \partial D(x_1,r_{n-1})$, 
\begin{align*}
P^{x_i}(T_{y_0}<T_{\partial D(x_1,r_n)} )=\frac{1 +o(1) }{n }. 
\end{align*}
Therefore, to achieve uniformity in
$\vec{x}\in \hat{\mathcal{E}}_{\delta}[A]$, 
\begin{align*}
&\bigg|\frac{2 }{\pi }(\log r_n - \log d(x_i,x_l)^++o(1))
-a_{i,l}\frac{2n\log n }{\pi} \bigg|\\
\le&\max \frac{2 n\log n }{\pi}| b_{i,l}-a_{i,l}|
= o(1) n \log n,
\end{align*}
where the above maximum is over $ b_{i,l}= a_{i,l}+o(1)$ with
 $1\le i,l \le j$. 
In addition, as per Remark \ref{prop4**}, to achieve uniformity in 
$\vec{x}\in \hat{\mathcal{E}}_{\delta}[A]$, 
\begin{align}\notag
&\bigg|\chi \bigg(\bigg(\frac{2 }{\pi }(\log r_n - \log d(x_i,x_l)^++O(1))\bigg)_{1\le i,l \le j} \bigg)
-\frac{\pi}{2 n\log n}\chi(A)\bigg|\\
\label{geoB1}
= &\frac{o(1)}{n \log n}.
\end{align}
%Note that $ \chi( \tilde{A}_{\beta})6n\log n/\pi=\chi(\pi/(6n\log n)\tilde{A}_{\beta})$. 
Therefore, if we substitute $T_{\partial D(x_1,r_n)}$ and $y$  for $\tilde{\tau}$ and $y_0$ in (\ref{hha*}), (\ref{gg2*}) yields
\begin{align}\label{j*2}
\sum_{l=1}^jP^{x_l}(T_{y_0} < T_X\wedge T_{\partial D(x_1, r_n )} )
=\frac{1 +o(1)}{n }\frac{\pi}{2 n\log n}\chi(A).
\end{align}
Note that as per (\ref{g2*}), we obtain
\begin{align*}
P^{y_0}(T_{y_0} < T_X\wedge T_{\partial D(x_1, r_n)} )
=1-\frac{\pi+o(1)}{2 n\log n}.
\end{align*}
Subsequently, 
\begin{align*}
&P^{y_0}( T_X<T_{\partial D(x_1, r_n)} )\\
=&\sum_{i=0}^{\infty} P^{y_0}(T_{y_0} <T_X \wedge T_{\partial D(x_1, r_n)}  )^i
P^{y_0}(T_X =  T_{y_0}\wedge T_X\wedge T_{\partial D(x_1, r_n)} )\\
=&\frac{1}{1- P^{y_0}(T_{y_0} <T_X \wedge T_{\partial D(x_1, r_n)} )}
P^{y_0}(T_X < T_{y_0}\wedge T_{\partial D(x_1, r_n)})\\
=&\frac{ 2(1+o(1)) n \log n}{\pi}
P^{y_0}(T_X < T_{y_0}\wedge T_{\partial D(x_1, r_n)} )\\
=&\frac{2(1+o(1))n \log n}{\pi}
\sum_{l=1}^jP^{x_l}(T_{y_0} < T_X\wedge T_{\partial D(x_1, r_n)} ).
\end{align*}
The last equality comes from the time-reversal of a simple random walk. 
Therefore, in view of (\ref{j*2}), we have (\ref{res}). 
\end{proof}

\begin{proof}[Proof of Proposition \ref{prop8*}]
Note that we only have to show the result for a sequence $K_n$, 
because $b \in [1,3]$ is arbitrary and thus $K_n$ covers all sufficiently large integers. 
Fix $0<\delta_1<\alpha$. 
As per $(3.19)$ in \cite{Dembo2}, 
there exist $c>0$ and $\delta>0$ such that for any $0<\alpha<1$, $n\in \mathbb{N}$ and $\vec{x}\in \hat{\mathcal{E}}_{\delta}[A]$,
\begin{align*}
P\bigg(\frac{4\alpha}{\pi}(K_n\log K_n)^2 <\mathcal{R}_n^{x_1}(\alpha-\delta_1)\bigg)
\le c^{-1}\exp(-cn^2\log n).
\end{align*} 
We find that for any $n\in \mathbb{N}$ 
\begin{align}\notag
&P(\vec{x} \in \mathcal{L}_{K_n}(\alpha)^j )\\
\notag
\le &
P(T_X >\mathcal{R}^{x_1}_n(\alpha-\delta_1) )
+P\bigg(\frac{4\alpha}{\pi}(K_n\log K_n)^2 <\mathcal{R}_n^{x_1}(\alpha-\delta_1) \bigg)\\
\label{sr1}
\le 
&CK_n^{-2(\alpha-\delta_1) \chi(A)+\epsilon/2}
 +c^{-1}\exp(-cn^2\log n).
\end{align} 
Note that as per Lemma $4.1$ in \cite{Dembo2},  for any $0<\delta_2<1-\alpha$, 
there exists $c>0$ such that for any $n\in \mathbb{N}$ and $\vec{x}\in \hat{\mathcal{E}}_{\delta}[A]$,
\begin{align*}
P\bigg(\frac{4\alpha}{\pi}(K_n\log K_n)^2 >\mathcal{R}_n^{x_1}(\alpha+\delta_2)\bigg)
\le c^{-1}\exp(-cn^2\log n).
\end{align*} 
Then, we have that 
for $\vec{x}\in \hat{\mathcal{E}}_{\delta}[A]$ with $x_1 \in D(0,n/10)^c$,
\begin{align}\notag
&P(\vec{x} \in \mathcal{L}_{K_n}(\alpha)^j )\\
\notag
\ge & P(T_X >\mathcal{R}^{x_1}_n(\alpha+\delta_2) )
-P\bigg(\frac{4\alpha}{\pi}(K_n\log K_n)^2 >\mathcal{R}_n^{x_1}(\alpha+\delta_2)\bigg)\\
\label{sr3}
\ge &cK_n^{-2(\alpha+\delta_2) \chi(A)-\epsilon/2}
 -c^{-1}\exp(-cn^2\log n).
\end{align}
Therefore, if we select sufficiently small $\delta_1$, $\delta_2>0$ for $\epsilon>0$, 
we obtain Proposition \ref{prop8*}. 
\end{proof}

%%%%%%%%%%%%%%%%%%%%%%%%%%%%%%%%%%%%%%%%%%%%%%%%%%%%%%%%%%%%%%%%%%%%%%%%%%%%%%%%%%%%%%%%%%%%%%%%%%%%
%\subsection{Proof of Theorem \ref{h2}}\label{UB in Theoh2}
%%%%%%%%%%%%%%%%%%%%%%%%%%%%%%%%%%%%%%%%%%%%%%%%%%%%%%%%%%%%%%%%%%%%%%%%%%%%%%%%%%%%%%%%%%%%%%%%%%%%%%%%%%%%%%%%%%%%%%%%%%%%%%%%%%%%%

\subsection{Proof of Theorem \ref{h2}}\label{Theoh2*}
\begin{proof}[Proof of the upper bound  in Theorem \ref{h2}]
Fix $0<\beta<1$. Propositions \ref{prop8*} and \ref{prop9*} yield that 
for any $\epsilon>0$, there exists $C>0$ such that for any $0<\alpha <1$, and $n\in \mathbb{N}$, 
\begin{align}\label{kj1}
\sum_{\vec{x}\in \mathcal{E}[ (n^{\eta}) , (n^\beta) ]}P(\vec{x} \in \mathcal{L}_n(\alpha)^j)
\le 
C n^{\hat{\rho}_j(\alpha, \beta)+\epsilon}.
\end{align}
Now we extend the result for  ``$\mathcal{E}[ (n^{\eta }), (n^\beta)]$'' to ``$\mathcal{E}[ (0), (n^\beta)]$'' 
by performing induction on $j \in \mathbb{N}$. 
We assume that for any $\epsilon>0$, there exists $C>0$ such that for any $n\in \mathbb{N}$, 
\begin{align*}
\sum_{(x_1,\ldots,x_{j-1})\in \mathcal{E}[ (0) , (n^\beta) ]}P(x_1,\ldots,x_{j-1} \in \mathcal{ L}_n(\alpha))
\le 
C n^{\hat{\rho}_j(\alpha, \beta)+(j-1)\epsilon}.
\end{align*}
For $j=1$, 
according to Proposition \ref{prop8*}, it is trivial that  
\begin{align*}
\sum_{x\in \mathbb{Z}^2_n}P(x\in  \mathcal{ L}_n(\alpha))
\le 
C n^{\hat{\rho}_1(\alpha,\beta)+\epsilon}.
\end{align*}
Let us assume that the claim holds for $j-1$ with $j\ge 2$. 
We show that the claim holds for $j$. 
For any $\epsilon>0$,  we select $\eta>0$ with $2\eta<\epsilon$ and $n_0$ given in Proposition \ref{kii}.  
Therefore, according to (\ref{kj1}), Lemma \ref{prop7*}, and induction, we obtain that for any $n \ge n_0$, 
\begin{align*}
&E[|\{ \vec{x}\in \mathcal{L}_n(\alpha)^j :
d(x_i,x_l)\le n^{\beta}  \text{ for any }1\le i,l \le j \}|] \\
=&
\sum_{ \vec{x}\in\mathcal{E}[(0),(n^\beta)]}
P(\vec{x} \in \mathcal{L}_n(\alpha)^j)\\
\le&  \sum_{ \vec{x}\in \mathcal{E}[ (n^{\eta}) , (n^\beta) ]}
P(\vec{x} \in \mathcal{L}_n(\alpha)^j)\\
& + \sum_{ (x_1,...,x_{j-1})\in  \mathcal{E}[(0),(n^\beta)](j-1)}
P(x_1,\ldots,x_{j-1}\in \mathcal{L}_n(\alpha))Cn^{2 \eta}\\
\le&  C n^{\hat{\rho}_j(\alpha, \beta)+\epsilon} 
+C n^{\hat{\rho}_{j-1}(\alpha, \beta)+(j-1)\epsilon+2\eta}
\le C n^{\hat{\rho}_j(\alpha, \beta)+j\epsilon}.
\end{align*}
As it suffices to show it for all sufficiently large $n\in \mathbb{N}$, we obtain the desired result. 
\end{proof}

%%%%%%%%%%%%%%%%%%%%%%%%%%%%%%%%%%%%%%%%%%%%%%%%%%%%%%%%%%%%%%%%%%%%%%%%%%%%%%%%%%%%%%%%%%%%%%%%%%%%%%%%%%%%%%%%%%%%%%%%%%%%%%%%%%%%%%

We write $A_r^{(j)}$ for $(a_{i,l})_{1\le i,l \le j}$ if  
$a_{i,i}=1$ and $a_{i,l}=r$ for $1\le i\neq l\le j$ and $1-\beta \le r \le 1- \eta$. 
Note that $A_{r}^{(j)}\in \mathcal{M}_j$. 
In addition, $A_{r}^{(1)}$ is independent of $r$, and therefore, 
we sometimes write $A^{(1)}$.  
%%%%%%%%%%%%%%%%%%%%%%%%%%%%%%%%%%%%%%%%%%%%%%%%%%%%%%%%%%%%%%%%%%%%
\begin{proof}[Proof of the lower bound in Theorem \ref{h2}]
It is trivial that 
$\chi(A^{(j)}_{1-l})
=j/(1+(j-1)(1-l))$. 
Fix $0<\beta<1$ and $\epsilon>0$ and pick $\delta>0$ in Proposition  \ref{prop8*}. 
If we consider $\vec{x} \in \mathcal{E}[(n^{l}),(5jn^{l})]$ with $x_1\in D(0,n/10)^c$ 
for $0<\eta<l<1$, then Proposition \ref{prop8*} yields that 
for any $0<\alpha<1$ and all sufficiently large $n\in \mathbb{N}$ with $5jn^{l}\le 2^j n^{l+\delta}$, 
\begin{align*}
P(\vec{x} \in \mathcal{L}_n(\alpha)^j)
\ge \exp\bigg(-\frac{2j\alpha \log n}{1+(j-1)(1-l)}+o(\log n)\bigg).
\end{align*}
Let 
\begin{align*}
R:=
\bigg\{&\vec{x}: x_1\in \mathbb{Z}_n^2\cap D\bigg(0,\frac{n}{10}\bigg)^c,\\
&x_i \in x_1+(0, 4(i-1)n^l )+D(0,n^l )\text{ for any }2\le i\le j \bigg\}.
\end{align*} 
Note that there exists $c>0$ such that for any $n\in \mathbb{N}$, 
\begin{align*}
|R|\ge cn^{2+2(j-1)l}. 
\end{align*}
In addition, 
\begin{align*}
\mathcal{E}[(n^l),(5jn^l)] \supset R. 
\end{align*}
Therefore, Proposition \ref{prop9*} and the simple computation yield that for $\eta<s<1$
\begin{align*}
&\sum_{\vec{x} \in \mathcal{E}[(n^{l}),(5jn^l)]}
P(\vec{x} \in \mathcal{L}_n(\alpha)^j)\\
\ge &cn^{2+2(j-1)l} \times
\exp \bigg(-\frac{2j\alpha \log n}{1+(j-1)(1-l)}+o(\log n)\bigg). 
\end{align*}
As $2+2(j-1)l-2\alpha j /(1+(j-1)(1-l))
|_{l=(1+(1-\sqrt{j\alpha})/(j-1)) \wedge \beta }=\hat{\rho}_j(\alpha,\beta)$ 
and $\eta$ is arbitrary, we obtain the result. 
\end{proof}
%%%%%%%%%%%%%%%%%%%%%%%%%%%%%%%%%%%%%%%%%%%%%%%%%
\section{Matrix argument}\label{matrix}
In this section, our goal is to arrive at Proposition \ref{prop9*}, which is used in 
the proof of the upper bound in Theorem \ref{h2}.   
We use only Propositions \ref{prop9*} and \ref{prop4*} in 
the proof of Theorem \ref{h2}.  
%We estimate them by using matrix's arguments. 
To show Proposition \ref{prop9*}, we prepare some propositions and lemmas in Section \ref{cl of matrix} and provide proofs in Section \ref{various pro}. 
\subsection{Claims}\label{cl of matrix}
We first establish results that yield the properties of matrices in $\mathcal{M}_j$ and then those that link the properties of $\mathcal{M}_j$ with the main results. 
%Our argument in later is independent of the arrangement of the elements included in  $\Lambda\subset \mathbb{N}$. 
%Then, we consider $(a_{i,l})_{ i,l \in  \Lambda}$ as the matrix for some arrangement. 
Note that $(c)$ in the definition of $\mathcal{M}_j$ in Section \ref{some in Theoh2} can be rewritten as 
\begin{align*}
(d)&\text{ for any }1\le  i, l, p \le j\text{ with }i \neq l, l\neq p,p\neq i,\\
&\text{ it holds that  }a_{i,l}< a_{i,p} \Rightarrow a_{l,p}=a_{i,l}
\end{align*}
assuming $(a)$ and $(b)$. 
Hereafter, we simply write $A$ for $(a_{i,l})_{1\le i,l \le j}$.

Now, we introduce propositions that provide the properties of $\mathcal{M}_j$. 
%%%%%%%%%
For $j_k \in \mathbb{N}$, let $A_k:=(a^{(k)}_{i,l})_{1\le i,l \le j_k}\in \mathcal{M}_{j_k}$ 
($\forall k=1,\ldots,m$) and 
$j=\sum_{k=1}^m j_k$. 
For the injective function $\sigma_k : \{1,\ldots, j_k\} \to \{1, \ldots ,j \}$ ($\forall k=1,\ldots,m$) with
$ \cup_{k=1}^m \mathrm{Im \; } \sigma_k =\{1,\ldots, j \}$ 
and $s\le \min \{ a_{i,l}^{(k)} \mid  k \in \{ 1,\ldots,m \}, i,l \in \{ 1, \ldots, j_k \} \}$, 
we let $A=
A_1^{\sigma_1} \boxplus_s ... \boxplus_s A_m^{\sigma_m}$ if
\begin{align*}
a_{i,l}:=
\begin{cases}
a^{(k)}_{\sigma_k^{-1}(i), \sigma_k^{-1}(l) }
& (\forall i,l \in \mathrm{Im \; } \sigma_k, k=1,\ldots,m ),\\
s&
\mbox{otherwise.}
\end{cases}
\end{align*}
Note that definitions yield
$A_1^{\sigma_1} \boxplus_s ... \boxplus_s A_m^{\sigma_m} \in \mathcal{M}_j$ and 
$\min_{1\le i,l \le j}a_{i,l}=s$. 
%We define the equivalence class $\mathcal{M}_j/\sim$. 
We define $(a_{i,l})_{1\le i,l \le j} \cong ({a'}_{i,l})_{1\le i,l \le j} $ 
if there exists a bijective function $\sigma: \{1,\ldots, j \} \to \{1, \ldots, j \}$ 
such that $a_{\sigma(i),\sigma(l)}={a'}_{i,l}$ for any $1\le i,l \le j$. 
\begin{prop}\label{prop1*} 
It holds that for any $j\ge2$ with $j\in \mathbb{N}$, $A \in \mathcal{M}_j$ satisfies the following: 
there exist $A_k\in \mathcal{M}_{j_k}$, $\sigma_k$ for $k=1,\ldots,m$ with $m\ge2$ such that 
 $A=A_1^{\sigma_1}\boxplus_s ... \boxplus_s A_m^{\sigma_m}$, 
 where $s< \min \{ a_{i,l}^{(k)} \mid k \in \{ 1,\ldots,m \}, i,l \in \{ 1, \ldots, j_k \} \}$. 
\end{prop}
\begin{rem}\label{dec1}
We call $A_1^{\sigma_1}\boxplus_s ... \boxplus_s A_m^{\sigma_m}$ 
the maximal decomposition of $A$ 
if $s< \min \{ a_{i,l}^{(k)} \mid k \in \{ 1,\ldots,m \}, i,l \in \{ 1, \ldots, j_k \} \}$. 
We show that 
if ${A'}_1^{{\sigma'}_1} \boxplus_s ... \boxplus_s {A'}_{m'}^{{\sigma'}_{m'}}$ 
is the maximal decomposition of $A$, 
then $m=m'$ and there exists a bijective function $\tilde{\sigma}: \{1,\ldots, j \} \to \{1, \ldots ,j \}$ such that 
$A_{\tilde{\sigma}(k)} \cong {A'}_{k}$ in Remark \ref{rem4}. 
Therefore, the maximal decomposition is uniquely determined in a certain sense. 
The maximal decomposition corresponds to clustering a $j$-tuple point 
by the maximal distance in the ultrametric space.
\end{rem}
\begin{prop}\label{prop2*} 
Any element included in $ \mathcal{M}_j$ is a regular matrix. 
In other words, for any $A\in \mathcal{M}_j$, 
there exists a unique solution $y_1,\ldots,y_j$ such that 
$$A\vec{y}^T=\vec{1}^T,$$
where $\vec{y}:=(y_1,\ldots,y_j)$ and $\vec{1}:=(1,\ldots,1)$. 
\end{prop}
\begin{rem}
References \cite{mar} and \cite{nab} demonstrated that a strictly symmetric ultrametric matrix is a regular matrix, 
and therefore, that 
the desired result had already been obtained. 
However, we provide another proof because the argument is used later. 
\end{rem}
%Now, we assume that Proposition \ref{prop1*} holds. 
%We want to define a functional 
%$\Xi: \cup_{j\in \mathbb{N}} \mathcal{M}_j \to \mathbb{R}$ on the basis of Proposition \ref{prop1*}. 
%For $A \in \mathcal{M}_j$, there is  $G$ and $H$ such that $A$  is expressed as in (c') in Proposition \ref{prop1*}. 

We define $\Xi$ inductively as follows: 
for $A \in \mathcal{M}_j$ whose maximal decomposition 
is $A_1^{\sigma_1} \boxplus_s ... \boxplus_s A_m^{\sigma_m}$, 
\begin{align*}
\Xi (A) : =\sum_{k=1}^m \Xi (A_k) +(m-1)(1-s), 
\end{align*} 
where $\Xi(A):=0$ for $A\in \mathcal{M}_1$. 
\begin{rem}
We simultaneously show the claim that $\Xi$ is well-defined and $\Xi(A)=\Xi(A')$ for $A \cong A'$ by performing induction on $j\in \mathbb{N}$. 
This is trivial for $j=1$. 
We assume the claim for $1,\ldots,j-1$ and show the claim for $j$. 
Subsequently, Remark \ref{dec1} and the assumption yield 
that $\Xi$ is well-defined for $j$. 
Note that if $A \cong A'$ holds 
and $A_1^{\sigma_1}\boxplus_s ... \boxplus_s A_m^{\sigma_m}$ 
is the maximal decomposition of $A$, 
there exists ${\sigma'}_k$ for $1\le k \le m$ such that 
$A' =A_1^{{\sigma'}_1} \boxplus_s ... \boxplus_s A_m^{{\sigma'}_m}$. 
Therefore, we obtain $\Xi(A)=\Xi(A')$ for $j$ and retain the claim. 
%$\Xi(A)$ denotes the number of points whose distances of pairs are determined by $A$. 
\end{rem}

Next, we observe the additional properties of the matrix included in $\mathcal{M}_j$.  
\begin{prop}\label{prop5*} 
For any $r\le j-1$
\begin{align*}
\min_{A \in {\Xi}^{-1}(\{r\})}
 \chi(A)
=\chi \bigg(A^{(j)}_{1-r/(j-1)}\bigg)
=\frac{j}{j-r}.
\end{align*}
\end{prop}

We provide the following lemmas concerning the configuration of points, 
which link the matrix argument with Proposition \ref{prop9*}.  
To describe our goal in this section, we give the following lemma. 
Note that 
$ (j-1)\eta \le \Xi(A)\le (j-1)\beta$ for  $A\in \mathcal{M}_j^{\beta,\eta}$. 
\begin{lem}\label{kii*}
For any $\epsilon>0$ and $0<\delta< e^{-j} \epsilon $,  
there exists $C>0$ such that for any $0<t\le (j-1)\beta$, 
$A \in \Xi^{-1}(\{t\})$ and  $n\in \mathbb{N}$, 
$$|\hat{\mathcal{E}}_{\delta}[A]|\le Cn^{2t+2+\epsilon/2}.$$
\end{lem}

To introduce Proposition \ref{prop9*},  
we prepare the following notation. 
%Recall that $\chi(A)$ is the summation over all the elements of $A^{-1}$. 
%$F_{h,\beta}(\gamma)=F(h,\beta,\gamma):=h\gamma^2\beta+(1-\gamma\beta)^2/(1-\beta)$. 
%Let $\gamma_h=\gamma(h):=1/(h(1-\beta)+\beta)$, where $F_{h,\beta}(\gamma)$ is minimized. 
As per Propositions \ref{kii} and \ref{prop2*}, 
for $\delta>0$, $\vec{x} \in \mathcal{E}[ (n^{\eta }), (n^\beta)] $, and  $n\ge n_0$, we can set 
\begin{align*}
h=h_{\delta}(\vec{x}):=\inf \{\chi (B): \vec{x}\in\hat{\mathcal{E}}_{\delta}[B], 
B \in   \mathcal{M}_j^{\beta, \eta} \}.
\end{align*}
%Note that for any $A \in   \mathcal{M}_j^{\beta, \eta}$, 
%$((\beta -1+a_{i,l})/\beta)_{1\le i,l\le j}\in \mathcal{M}_j$ holds. We fix $0<\alpha<1$. 
\begin{prop}\label{prop9*}
For any $\epsilon>0$ and $0<\delta< e^{-j} \epsilon$, there exists $C>0$ such that for any $n\in \mathbb{N}$,
\begin{align*}%\notag
 \sum_{\vec{x}\in \mathcal{E}[ (n^{\eta }), (n^\beta)]}
n^{-2\alpha  h_{\delta}(\vec{x})}
\le C
n^{\hat{\rho}_j(\alpha,\beta)+\epsilon }.
\end{align*}
\end{prop}
\begin{proof}[Proof of Proposition \ref{prop9*}]
Note that Proposition \ref{prop5*} implies 
\begin{align*}
\min_{\substack{\vec{x}\in\hat{\mathcal{E}}_{\delta}[A], 
\\A\in {\Xi}^{-1}(\{t\})}}
h_{\delta}(\vec{x})
=\min_{A\in  {\Xi}^{-1}(\{t\})}\chi(A)
=\frac{j}{j-t}.
\end{align*} 
For any $\delta>0$ and $\epsilon>0$, there exist 
$C':=\lceil (\beta-\eta)/\delta \rceil$ and $C>0$ such that for any $n\ge n_0$, 
the left-hand side of the desired formula is bounded by
\begin{align*}
 &(C')^j \max_{0\le t\le (j-1)\beta}
\max_{ A \in \Xi^{-1}(\{t\})}
\sum_{\vec{x}\in \hat{\mathcal{E}}_{\delta}[A]}
n^{-2\alpha  h }\\
\le& C\max_{0\le t\le (j-1)\beta}
n^{2t+2+\epsilon/2} 
\max_{\substack{\vec{x}\in\hat{\mathcal{E}}_{\delta}[A], 
\\A\in {\Xi}^{-1}(\{t\})}}
n^{-2\alpha h}\\
\le& C\max_{0\le t\le (j-1)\beta}
n^{2t+2+\epsilon}
n^{-2\alpha  j/(j-t) }.
\end{align*}
The first inequality comes from Lemma \ref{kii*} and 
the last one comes from Proposition \ref{prop5*}. 
Therefore, we have
\begin{align*}
 \max_{0\le t\le (j-1)\beta}
2t+2-\frac{2\alpha j}{j-t}=\hat{\rho}_j(\alpha,\beta).
\end{align*}
Because it is sufficient to show the claim for $n\ge n_0$, we obtain the desired result.
\end{proof}

\subsection{Proofs of various propositions and lemmas}\label{various pro}
In this subsection, we provide proofs of the propositions and lemmas 
that are introduced in Sections \ref{some in Theoh2} and \ref{cl of matrix}. 

First, we provide the proof of Proposition \ref{prop1*}. 
\begin{proof}[Proof of Proposition \ref{prop1*}]
Let $s:=\min_{1\le i,l \le j} (a_{i,l})_{1\le i,l \le j}$. 
We define $k \sim k'$ if 
$a_{k,k'}>s$. 
First, we show that $\sim$ constructs an equivalence class. 
Note that reflexive and symmetric relations are trivial owing to the definition of $\mathcal{M}_j$, 
and therefore, 
we show a transitive relation. 
Let us assume that $k_1\sim k_2$ and $k_2\sim k_3$. 
The definition of $\mathcal{M}_j$ yields 
$a_{k_1,k_3}\ge \min \{a_{k_1,k_2} ,a_{k_2,k_3} \}>s$. 
Therefore, we obtain $k_1\sim k_3$ and that $\{1,\ldots,j\}/\sim$ is an equivalence class. 
Next, we show the claim. 
If $|\{1,\ldots,j\}/\sim |=m$, we let $G_1,\ldots,G_m$ be  elements in $\{1,\ldots,j\}/\sim$ 
and $j_k$ be $|G_k|$ for $1\le k\le m$.  
For any $1\le k\le m$, we select some bijective function $\sigma_k:\{1,\ldots,j_k\}\to G_k$.  
We set $A_k:=(a^{(k)}_{i,l})_{1\le i,l \le j_k}$ such that $a^{(k)}_{\sigma_k^{-1}(i), \sigma_k^{-1}(l) }=a_{i,l}$ 
for $1\le i,l \le j_k$. 
Then, $A=A_1^{\sigma_1}\boxplus_s ... \boxplus_s A_m^{\sigma_m}$ and 
$s< \min \{ a_{i,l}^{(k)} \mid k \in \{ 1,\ldots,m \}, i,l \in \{ 1, \ldots, j_k \} \}$ holds. 
Therefore, we obtain the desired result.
\end{proof}

\begin{rem}\label{rem4}
If $A_1^{\sigma_1}\boxplus_s ... \boxplus_s A_m^{\sigma_m}$ is 
the expression of the maximal decomposition of $A$, 
it is trivial that $\{\mathrm{Im \; } \sigma_k: 1\le k\le m\}=\{1,\ldots,j\}/\sim$ by the above proof. 
Then, it easily yields the uniqueness of the maximal decomposition and the claim in Remark \ref{dec1}. 
\end{rem}
%%%%%%%%%%%%%%%%%%%%%%%%%%%%%%%%%%%%%%%%%%%%%%%%%%%%%%%%%%%%%%%%%%%%%%%%%%%%%%%%%%%%%%%%%%%%%%%%%%%%%%%%%%%%%%%%%%%%%%%%%%%%%%%%%%%%%%%%
%If $A = A_1^{\sigma_1} \boxplus A_2^{\sigma_2}$ and we let$G := \{1,...,j_1 \}$, $H := \{1,...,j_2 \}$, $j_1=g$ and $j_2=h$, 
%we define $A_G := ( a_{i,l}^1 )_{1 \le i,l \le j_1}$ and $A_H := ( a_{i,l}^2 )_{1 \le i,l \le j_2}$. 
%%%%%%%%%%%%%%%%%%%%%%%%%%%%%%%%%%%%%%%%%%%%%%%%%%%%%%%%%%%%%%%%%%
%%%%%%%%%%%%%%%%%%%%%%%%%%%%%%%%%%%%%%%%%%%%%%%%%%%%%%%%%%%%%%%%%%
Hereafter, we assume that $A=A_1^{\sigma_1}\boxplus_s A_2^{\sigma_2}$ unless otherwise stated. 
Note that that $A_1^{\sigma_1}\boxplus_s A_2^{\sigma_2}$ is not always the maximal decomposition of $A$. 
Let $g:=|\mathrm{Im \;} \sigma_1|$ and $h:=|\mathrm{Im \;} \sigma_2|$. 
\begin{proof}[Proof of Proposition \ref{prop2*}]
We show the claim and present the following conditions $(A)$ and $(B)$: when $j=2$, 
\begin{align*}
&(A)\quad 1-s\sum_{i=1}^j y_i >0,\\
&(B)\quad y_i> 0 \quad \text{ for any }1\le i\le j. 
\end{align*}
When $j=1$, we change $(A)$ to $1-y_1 \ge0$. 
As per symmetry, we need to show the result for only the case in which  
$\mathrm{Im \;} \sigma_1=\{1,\ldots,g\}$ and $\mathrm{Im \;} \sigma_2=\{g+1,\ldots,j\}$.   
We show the results by performing induction on $j \in \mathbb{N}$. 
It is trivial that the claim holds for $j=1$ as $y_1=1$.  
Assuming that the claims $(A)$ and $(B)$ hold for $1,\ldots,j-1$, we show that the claims $(A)$ and $(B)$ hold for $j$.  
As per symmetry, this assumption yields that $A_1\in \mathcal{M}_g$ determines 
a unique solution $\vec{z}:=(z_1,\ldots,z_g)$  such that 
$A_1 \vec{z}^T=\vec{1}^T$ and 
$A_2 \in \mathcal{M}_h$ determines 
a unique solution $\vec{z}':=(z'_1,\ldots,z'_h)$ such that 
$A_2 \vec{z}'^T=\vec{1}^T$. 
In addition, the assumption of $(B)$ yields $z_i> 0$ for any $1\le i\le g$ and 
$z'_i>0$ for any $1\le i\le h$. 
Therefore, it holds that $z_i+s\sum_{1\le l \le g, l\neq i} z_l\le 1$ for any $1\le i\le g$ and 
$z_i+s\sum_{1\le l \le h, l\neq i} z'_l\le 1$ for any $1\le i\le h$, and we are able to derive the following equation:
\begin{align*}
\sum_{i=1}^g z_i \le \frac{g}{1+(g-1)s}, \quad 
 \sum_{i=1}^h z'_i \le \frac{h}{1+(h-1)s}. 
\end{align*}
There exists $c>0$ such that 
\begin{align}\label{asm*}
1-s^2 \sum_{i=1}^g z_i \sum_{i=1}^h z'_i  \ge c.
\end{align}
This comes from $s\le 1-\eta$. 
If we set 
\begin{align*}
y_l=&\frac{(1-s \sum_{i=1}^h z'_i )z_l }
{1-s^2 \sum_{i=1}^g z_i \sum_{i=1}^h z'_i }
\quad \text{ for any } 1\le l\le g,\\
y_l=&\frac{(1-s \sum_{i=1}^g z_i )z'_{l-g} }
{1-s^2 \sum_{i=1}^g z_i \sum_{i=1}^h z'_i }
\quad \text{ for any } g+1\le l\le j, 
\end{align*}
we obtain a solution such that 
$A\vec{y}^T=\vec{1}^T$. 
Therefore, we have proved the existence of the solution. 
Hereafter, we observe the properties of $y_1,\ldots,y_j$ 
by assuming their existence. 

First, we show $(B)$. 
According to Proposition \ref{prop1*}, it holds that 
\begin{align}\label{hh1}
&\sum_{i=1}^g (a_{l,i}y_i)+s\sum_{i=g+1}^j y_i=1
\quad \text{ for any } 1\le l\le g,\\
\label{hh2}
&s\sum_{i=1}^g y_i+\sum_{i=g+1}^j (a_{l,i}y_i)=1
\quad \text{ for any } g+1\le l\le j.
\end{align}
Therefore, as per the definition of $z_1,\ldots,z_g$, $z'_1,\ldots,z'_h$, we obtain 
\begin{align*}
\sum_{i=1}^g y_i=\bigg(1-s\sum_{i=g+1}^j y_i \bigg)\sum_{i=1}^g z_i, \quad
\sum_{i=g+1}^j y_i=\bigg(1-s\sum_{i=1}^g y_i \bigg)\sum_{i=1}^h z'_i.
\end{align*}
A simple computation and (\ref{asm*}) yield
\begin{align}\notag
\sum_{i=1}^g y_i&= \frac{\sum_{i=1}^g z_i -s \sum_{i=1}^g z_i \sum_{i=1}^h z'_i}
{1-s^2 \sum_{i=1}^g z_i \sum_{i=1}^h z'_i },\\
\label{kk1}
\sum_{i=g+1}^j y_i&=\frac{\sum_{i=1}^h z'_i-s \sum_{i=1}^h z'_i\sum_{i=1}^g z_i}
{1-s^2 \sum_{i=1}^g z_i \sum_{i=1}^h z'_i },
\end{align}
and therefore, 
\begin{align}\label{kk1+}
\sum_{i=1}^j y_i= 
\frac{\sum_{i=1}^g z_i+\sum_{i=1}^h z'_i -2s \sum_{i=1}^g z_i \sum_{i=1}^h z'_i }
{1-s^2 \sum_{i=1}^g z_i \sum_{i=1}^h z'_i }.
\end{align}
% and hence $\sum_{i=1}^j y_i$ is uniquely determined. 
If we let $\tilde{s}=\min _{ i,l \in \mathrm{Im \; } \sigma_1} a_{i,l}$, 
by assuming $(A)$ and setting $g\ge2$, we obtain 
\begin{align}\label{kk2}
1-\tilde{s}\sum_{i=1}^g z_i > 0. 
\end{align}
Because $\tilde{s}\ge s$ for $g\ge2$, (\ref{kk1}) and  (\ref{kk2}) yield
\begin{align*}
1-s\sum_{i=1}^g y_i =\frac{1-s \sum_{i=1}^g z_i }
{1-s^2 \sum_{i=1}^g z_i \sum_{i=1}^h z'_i }
\ge \frac{1-\tilde{s} \sum_{i=1}^g z_i }
{1-s^2 \sum_{i=1}^g z_i \sum_{i=1}^h z'_i }
> 0. 
\end{align*}
In addition, because $\tilde{s}> s$ for $g=1$,  (\ref{kk1}) and  (\ref{kk2}) yield
\begin{align*}
1-sy_1 =\frac{1-s z_1 }
{1-s^2 z_1 \sum_{i=1}^h z'_i }
>\frac{1-\tilde{s} z_1 }
{1-s^2 z_1 \sum_{i=1}^h z'_i }
\ge 0. 
\end{align*} 
As per the definition of $y_{g+1},\ldots,y_j$, $z'_1,\ldots,z'_h$ and (\ref{hh2}), we have 
$(y_{g+1},\ldots,y_j) = (1-s\sum_{i=1}^g y_i )\vec{z}'$. 
Subsequently, because we assume that the solution of $\vec{z}'$ satisfies $z'_i>0$ for any $1\le i\le h$, 
it holds that any solution $y_{g+1},\ldots,y_j$ satisfies 
$y_i>0$ for any $g+1\le i\le j$. 
In addition, as per the same above-mentioned argument, 
the definitions of $y_1,\ldots,y_g$, $z_1,\ldots,z_g$ and (\ref{hh1}) 
yield $y_i>0$ for any $1\le i\le g$, 
and therefore, $(B)$ holds. 

Secondly, we  show $(A)$. 
The fact that $y_i>0$ for any $1\le i\le j$ and $\sum_{i=1}^ja_{l,i}y_i=1$ for any $1\le l\le j$ yields $s\sum_{i=1}^j y_i< 1$ 
allows us to obtain the desired results.

Now, we turn to prove the uniqueness of the solution using the result of $(B)$ that we already obtained. 
In general, it is known that 
\begin{align*}
V:=\{\vec{y}:A\vec{y}^T=\vec{1}^T \}=x_0+\text{Ker}A,
\end{align*}
where $x_0$ is a characteristic solution for $A\vec{y}^T=\vec{1}^T$. 
As per the result of $(B)$, it holds that 
$\{v=(v_1,\ldots,v_j): v_i>0\}\supset V$. 
Because $V$ is a linear space, the equation Ker$A \neq \{0\}$ is contradictory. 
Subsequently, Ker$A= \{0\}$, and therefore, we have the desired claim. 
\end{proof}
%Now, to show Lemma \ref{divide*}, we define the surjective function $M\{\cdot, \cdot,s\}$: $\mathcal{M}_g\times \mathcal{M}_h \to \mathcal{M}_j$ with $g+h=j$. 
%If $A_{\Lambda}$ verifies conditions of $\tilde{\mathcal{M}}_j$ and $G=\{1,...,g\}$, $H=\{g+1,...,g+h(=j)\}$, we write $M\{A_G,A_H,s\}$ for $A_{\Lambda}$. 
%However, we conveniently use the following meaning of this symbol.  
%If we write $A=M\{(b_{i,l})_{1\le i,l \le g}, (\hat{b}_{i,l})_{1\le i,l \le h},s\}$ for $g+h=j$ and $A\in \mathcal{M}_j$, 
%we consider the case that
%\begin{align*}
%\begin{cases}
%a_{i,l}=\min_{1\le i,l \le j}a_{i,l} =s &\quad \text{ for } 1\le  i \le g, g+1\le  l \le j,\\
%a_{i,l}=b_{i,l}&\quad \text{ for }1\le  i,l \le g,\\
%a_{i,l}=\hat{b}_{i-g,l-g}&\quad \text{ for }g+1\le  i,l \le j.\\
%\end{cases}
%\end{align*}
%%%%%%%%%%%%%%%%%%%%%%%%%%%%%%%%%%%%%%%%%%%%%%%%%%%%%%%%%%%%%%%%%%
To show Proposition \ref{prop5*}, we use the two lemmas. 
We argue with the values of $\Xi$ and $\chi$ in Proposition \ref{prop5*} and the following lemmas. 
Because $\Xi$ and $\chi$ are independent of $\sigma_1$ and $\sigma_2$, 
we omit $\sigma_1$ and $\sigma_2$. 
For example, we write $A_1 \boxplus_s A_2$ for $A_1^{\sigma_1} \boxplus_s A_2^{\sigma_2}$. 

%%%%%%%%%%%%%%%%%%%%%%%%%%%%%%%%%%%%%%%%%%%%%%%%%%%%%%%%%%%%%%%%%%

%%%%%%%%%%%%%%%%%%%%%%%%%%%%%%%%%%%%%%%%%%%%%%%%%%%%%%%%%%%%%%%%%%

 \begin{lem}\label{r2} 
%For any $(a_{i,l})_{1\le i,l \le j}$, $(\overline{a}_{i,l})_{1\le i,l \le j}\in \mathcal{M}_j$ such that two sequences have same division as $G$, $H$ %and $s=\min_{1\le i,l \le j} a_{i,l}=\min_{1\le i,l \le j}\overline{a}_{i,l}$. 
Consider $A$, $\overline{A} \in \mathcal{M}_j$ 
such that $A=A_1\boxplus_s A_2$ and 
$\overline{A}=\overline{A}_1 \boxplus_s \overline{A}_2$. 
Subsequently, if
\begin{align*}
\chi (A_1 )\ge \chi (\overline{A}_1)\text{ and }
\chi (A_2 ) \ge \chi (\overline{A}_2 ),
\end{align*}
it holds that 
\begin{align*}
\chi(A)
\ge \chi (\overline{A}).
\end{align*}
\end{lem}
\begin{proof}
Note that it suffices to prove the case that  
$\chi(A_1)=\chi(\overline{A}_1)$ 
because we can prove the claim by repeating the same proof.  
Let 
\begin{align*}
g(t,b,c):=\frac{b+c-2tbc}{1-t^2bc},
\end{align*}
where $1-t^2bc>0$. 
It is found that $g$ monotonically increases in $c$ 
because a simple computation yields
\begin{align*}
\frac{\partial g}{\partial c}
=\frac{(1-tb)^2}{(1-t^2bc)^2}\ge 0.
\end{align*}
Note that  if we consider $A$ and $\vec{y}$ such that 
$A\vec{y}^T=\vec{1}^T$, 
$\chi (A)=\sum_{i=1}^j y_i$ holds; 
then, (\ref{kk1+}) yields
\begin{align}\label{dec**}
\chi (A)
=g(s,\chi (A_1), \chi (A_2)). 
\end{align}
Because $g(t,b,c)$ monotonically increases in $c$, 
the assumption yields the desired result. 
\end{proof}

%%%%%%%%%%%%%%%%%%%%%%%%%%%%%%%%%%%%%%%%%%%%%%%%%%%%%%%%%%%%%%%%%%

 \begin{lem}\label{r3} 
Consider $r\le j-1$, $0 \le \gamma \le \gamma_1 \le \gamma_2\le 1$ 
with  $r=(g-1)(1-\gamma_2) +(h-1)(1-\gamma_1)+1-\gamma$ and $g+h=j$ with $g$, $h\ge1$,  
which satisfy $A^{(g)}_{\gamma_2}\boxplus_{\gamma} A^{(h)}_{\gamma_1}\in \Xi^{-1}(\{r\})$. 
Then, fixing the values $\gamma_2$ and $r$, 
$\chi (A^{(g)}_{\gamma_2}\boxplus_{\gamma} A^{(h)}_{\gamma_1})$ 
is minimized at $\gamma=\gamma_1$. 
\end{lem}
\begin{proof}
When we fix the values $\gamma_2$ and $r$, we find that  
$r=\Xi(A^{(g)}_{\gamma_2}\boxplus_{\gamma} A^{(h)}_{\gamma_1})
=(g-1)(1-\gamma_2) +(h-1)(1-\gamma_1)+1-\gamma$ is a constant, and therefore, we obtain 
$(g-1)\gamma_2+(h-1)\gamma_1+\gamma$. 
Subsequently, if we set $p:=(g-1)\gamma_2+1$ and $q:=(h-1)\gamma_1+\gamma+1$, we find that
$p$ and $q$ are constants. 
Note that $0 \le \gamma \le (q-1)/h$. 
In addition, 
\begin{align*}
\chi(A^{(g)}_{\gamma_2})
=\frac{g}{(g-1)\gamma_2+1}, \quad
\chi(A^{(h)}_{\gamma_1})
=\frac{h}{(h-1)\gamma_1+1}.
\end{align*}
According to (\ref{dec**}), it holds that 
\begin{align*}
f(\gamma):= &\chi (A^{(g)}_{\gamma_2}\boxplus_{\gamma} A^{(h)}_{\gamma_1})\\
=&\frac{g(q-\gamma)+hp-2\gamma gh}
{( q-\gamma )p -\gamma^2 gh}. 
\end{align*}  
It suffices to show the claim that $f$ monotonically decreases in $0 \le \gamma \le  (q-1)/h$. 
A simple computation yields 
\begin{align*}
\frac{\partial f(\gamma)}{\partial \gamma}
=&\frac{((q-\gamma)p-\gamma^2gh)(-g-2gh)-
(g(q-\gamma)+hp-2\gamma gh )(-p-2\gamma gh)}
{((q-\gamma )p -\gamma^2 gh)^2}.
\end{align*}
Set
\begin{align}\label{l0}
\tilde{f}:
=-g^2h(1+2h)(\gamma -\frac{gq+hp}{g+2gh})^2
+\frac{h(gq+hp)^2}{(1+2h)} -2qpgh+hp^2.
\end{align}
Note that it holds that 
\begin{align}\label{l1}
 \frac{q-1}{h} \le 
\text{the apex (summit) of }\tilde{f}
=
\frac{(gq+hp)}
{g(1+2h)}
\end{align}
because  
$h(gq+hp)-(q-1)g(1+2h)=h(g+h)(1-\gamma_2) \ge 0$. 
The first inequality comes from $-(q-1)\ge -((h-1)\gamma_2+\gamma_2)= -h\gamma_2$ 
and the second one comes from $\gamma_2\le 1$.
Therefore, we obtain (\ref{l1}). 
In addition, we claim 
\begin{align}\label{l2}
\tilde{f}\bigg(\frac{q-1}{h} \bigg) \le 0
\end{align}
because $h\tilde{f}((q-1)/h)
=-(ph-g(q-1))(2gh-g(q-1)-hp)
\le 0$.  
The inequality comes from 
$ph-g(q-1)\ge gh\gamma_2 \ge0 $ and $2gh-g(q-1)-hp\ge 2gh-gh -hp\ge0$. 
Therefore, as per (\ref{l0}), (\ref{l1}), and (\ref{l2}), we obtain
 for $0 \le \gamma \le (q-1)/h$,  
$\tilde{f}(\gamma) \le 0$ 
and the desired result. 
\end{proof}

%%%%%%%%%%%%%%%%%%%%%%%%%%%%%%%%%%%%%%%%%%%%%%%%%%%%%%%%%%%%%%%%%%
\begin{proof}[Proof of Proposition \ref{prop5*}]
It is trivial that 
$\chi(A^{(j)}_{1-r/(j-1)} )
=j/(j-r)$ holds; therefore, we show only 
$\min_{A \in \Xi^{-1}(\{r\})}
 \chi(A)
=\chi (A^{(j)}_{1-r/(j-1)} )$. 
%By the symmetry, we only need to show the result for only the case $G=\{1,...,g\}$, $H=\{g+1,...,g+h(=j)\}$. 
We prove the result by performing induction on $j\in \mathbb{N}$. 
If $j=1$ or $2$, it is obvious that the claim holds. 
We assume that the claim holds for $1,\ldots, j-1$ 
and show the claim for $j$. 
For $g\wedge h=1$, 
Lemma \ref{r3} yields the desired result. 
It suffices to show the result for $j\ge4$ with $g$, $h\ge 2$.  

For any $r\le j-1$ and $A=A_1\boxplus_s A_2 \in \Xi^{-1} (\{ r\})$, 
we select $\gamma_1$, $\gamma_2$, and $\gamma$,
which satisfy $\gamma=s$, 
$\Xi(A_1)=\Xi(A^{(g)}_{\gamma_2}) =(g-1)(1-\gamma_2)$, and 
$\Xi(A_2)=\Xi(A^{(h)}_{\gamma_1}) =(h-1)(1-\gamma_1)$.  
Without loss of generality, we can assume that $\gamma_1\le \gamma_2$. 
Note that $A^{(g)}_{\gamma_2} \boxplus_{\gamma} A^{(h)}_{\gamma_1}  \in \Xi^{-1} (\{ r\})$. 
According to Lemma \ref{r2} and this assumption, we obtain
 \begin{align}\label{h*1}
\chi(A^{(g)}_{\gamma_2} \boxplus_{\gamma} A^{(h)}_{\gamma_1})
\le \chi(A).
\end{align}
In addition, we consider $\tilde{\gamma_1}$ satisfying $(h-1)\gamma_1+\gamma=h\tilde{\gamma}_1$. 
Note that $ A^{(g)}_{\gamma_2} \boxplus_{\tilde{\gamma_1}}  A^{(h)}_{\tilde{\gamma}_1} \in \Xi^{-1} (\{ r\})$ and 
$\gamma_1 \ge \tilde{\gamma}_1 \ge \gamma$. 
As per Lemma \ref{r3}, we obtain 
 \begin{align}\label{h*2}
\chi( A^{(g)}_{\gamma_2} \boxplus_{\tilde{\gamma_1}} A^{(h)}_{\tilde{\gamma}_1})
\le \chi(A^{(g)}_{\gamma_2} \boxplus_{\gamma} A^{(h)}_{\gamma_1}).
\end{align}
Note that for any $\sigma_1$ and $\sigma_2$, we can select $\sigma_3$, $\sigma_4$, $\sigma_5$, and $\sigma_6$ such that 
$$(A^{(g)}_{\gamma_2})^{\sigma_1} \boxplus_{\tilde{\gamma_1}} (A^{(h)}_{\tilde{\gamma}_1})^{\sigma_2}
=((A^{(g)}_{\gamma_2})^{\sigma_3} \boxplus_{\tilde{\gamma}_1} (A^{(h-1)}_{\tilde{\gamma}_1})^{\sigma_4})^{\sigma_5} \boxplus_{\tilde{\gamma_1}} (A^{(1)})^{\sigma_6}.$$ 
In addition, we consider $\tilde{\gamma}_2$ satisfying $ (g-1)\gamma_2+h\tilde{\gamma}_1 =(j-2)\tilde{\gamma}_2+\tilde{\gamma}_1$. 
Note that $ A^{(j-1)}_{\tilde{\gamma}_2} \boxplus_{\tilde{\gamma_1}}  A^{(1)}  \in \Xi^{-1} (\{ r\})$ and 
$\gamma_2\ge \tilde{\gamma}_2\ge \tilde{\gamma}_1$. 
According to Lemma \ref{r2} and the assumption, we obtain
 \begin{align}\label{h*3}
\chi( A^{(j-1)}_{\tilde{\gamma}_2} \boxplus_{\tilde{\gamma_1}} A^{(1)})
\le \chi(A^{(g)}_{\gamma_2} \boxplus_{\tilde{\gamma_1}} A^{(h)}_{\tilde{\gamma}_1}).
\end{align}
Finally, Lemma \ref{r3} yields
 \begin{align}\label{h*4}
\chi(A^{(j)}_{1-r/(j-1)})
\le \chi( A^{(j-1)}_{\tilde{\gamma}_2} \boxplus_{\tilde{\gamma_1}} A^{(1)}).
\end{align}
Note that $A^{(j)}_{1-r/(j-1)} \in \Xi^{-1} (\{ r\})$. 
Therefore, as per (\ref{h*1}), (\ref{h*2}), (\ref{h*3}), and (\ref{h*4}), 
we obtain the desired result. 
\end{proof}

%%%%%%%%%%%%%%%%%%%%%%%%%%%%%%%%%%%%%%%%%%%%%%%%%%%%%%%%%%%%%%%%%%

%%%%%%%%%%%%%%%%%%%%%%%%%%%%%%%%%%%%%%%%%%%%%%%%%%%%%%%%%%%%%%%%%%
%\subsection{Preliminary for main results}\label{pre}
%To show Proposition \ref{prop9*}, 
%we use the following two lemmas. 

%%%%%%%%%%%%%%%%%%%%%%%%%%%%%%%%%%%%%%%%%%%%%%%%%%%%%%%%%%%%%%%%%%%%%%%%%%%

\begin{proof}[Proof of Lemma \ref{kii*}]
%To prove the claim, for  $0<t\le (j-1)\beta$
We prove the claim by performing induction on $j \in \mathbb{N}$. 
Because $\hat{\mathcal{E}}_{\delta}[A]=\mathbb{Z}_n^2$ for $j=1$ and 
$|\hat{\mathcal{E}}_{\delta}[A]| \le |\mathbb{Z}_n^2|\times Cn^{2 t+2\delta}$ for $j=2$ and $A\in  \Xi^{-1}(\{t\})$, 
it is obvious that the desired result holds for $j=1,2$. 
%Then we show the result for $j\ge 3$. 
We assume that the result holds for $1,\ldots,j-1$ with $j\ge 3$ and show the result for $j$. 
It suffices to prove that for any $\epsilon>0$, $0<\delta<\epsilon $, and  $L<\infty$, 
there exists $C>0$ such that 
for any $x\in \mathbb{Z}_n^2$, $t\le (j-1)\beta$, 
$A \in \mathcal{M}_j^{\beta,\eta} \cap \Xi^{-1}(\{t\})$, and  $n\in \mathbb{N}$ 
\begin{align}\label{MM}
| \mathcal{E}[(0), (L n^{1-a_{i,l}+\delta} )_{1\le i,l \le j}] |
\le Cn^{2t+2+\exp(j) \epsilon/2}.
\end{align}
First, we show the claim for the case that $g\wedge h=1$. 
Without loss of generality, we only prove it for $h=1$. 
Let $t_1:=\Xi(A_1)$ and $(a_{i,l}^1)_{1\le i,l \le j-1}:=A_1$. 
Note that $t=\Xi(A)=\Xi(A_1)+1-s=t_1+1-s$. 
Then, for any $0<\delta<\epsilon$, there exists $C>0$ such that for any $n\in \mathbb{N}$, 
\begin{align}\label{xx**}
 |\mathcal{E}[(0), (L n^{1-a_{i,l}^1+\delta} )_{1\le i,l \le j-1}] |\le Cn^{2t_1+2+\exp(j-1)\epsilon/2},
 \end{align}
 and therefore,
\begin{align*}
|\mathcal{E}[(0), (L n^{1-a_{i,l}+\delta} )_{1\le i,l \le j}]|  
\le &|\mathcal{E}[(0), (L n^{1-a_{i,l}^1+\delta} )_{1\le i,l \le j-1}]| \times Cn^{2-2s+2\delta} \\
\le & Cn^{2-2s+2t_1+2+\exp(j-1)\epsilon/2+2\delta} \\
\le &Cn^{2t+2+\exp(j)\epsilon/2}.  
\end{align*}
Then, we have proven the claim. 

Next, we show the claim for $j\ge 4$ and $g\wedge h \neq 1$. 
For $k \ge 2$, $x\in \mathbb{Z}_n^2$, and $L>0$, let
\begin{align*}
 \tilde{\mathcal{E}}_{\delta,x}[A]=
 &\tilde{\mathcal{E}}_{\delta,x}[A](k,L)\\
:=&\{(x_{2},\ldots,x_k):(x,x_2,\ldots,x_k) \in \mathcal{E}[(0), (L n^{1-a_{i,l}+\delta} )_{1\le i,l \le k}]  \}.
\end{align*}
Note that 
 \begin{align*}
\mathcal{E}[(0), (L n^{1-a_{i,l}+\delta} )_{1\le i,l \le j}] 
 \subset 
\{\vec{x} : x_1 \in \mathbb{Z}_n^2, (x_2,\ldots,x_j)\in \tilde{\mathcal{E}}_{\delta,x_1}[A]\},
\end{align*}
and therefore, 
 \begin{align*}
|\mathcal{E}[(0), (L n^{1-a_{i,l}+\delta} )_{1\le i,l \le j}] |
\le \sum_{x\in \mathbb{Z}_n^2}
| \tilde{\mathcal{E}}_{\delta,x}[A]|.
\end{align*}
Then, it suffices to prove that for any $\epsilon>0$, $0<\delta<\epsilon $, and $L<\infty$, 
there exists $C>0$ such that 
for any $x\in \mathbb{Z}_n^2$, $t\le (j-1)\beta$, 
$A \in \mathcal{M}_j^{\beta,\eta} \cap \Xi^{-1}(\{t\})$, and  $n\in \mathbb{N}$ 
\begin{align*}
| \tilde{\mathcal{E}}_{\delta,x}[A]|\le Cn^{2t+\exp(j) \epsilon/2}.
\end{align*}
%By the symmetry, we only need to show the result for the case $A$ has the divisions $G=\{1,...,g\}$, $H=\{g+1,...,g+h(=j)\}$. 
%It is obvious that the claim holds for $j=1,2$. 
%To show that the claim holds for $j$, consider $ \tilde{\mathcal{E}}_{\delta,x}[A] \in \tilde{\mathcal{H}}_{\delta,t,x}$. 
For any $A$, let $t_1:=\Xi(A_1)$ and $t_2:=\Xi(A_2)$, 
and therefore, $t=\Xi(A)=t_1+t_2+1-s$ holds. 
Note that it holds that
\begin{align*}
A_1 \in \mathcal{M}_g^{\beta,\eta} \cap \Xi^{-1}(\{t_1\}),\quad
A_2 \in \mathcal{M}_h^{\beta,\eta} \cap \Xi^{-1}(\{t_2\}). 
\end{align*}
Subsequently, as per the assumption, we find that for any $\epsilon>0$ and $0<\delta<\epsilon$,
there exists $C>0$ such that 
for any $x_1,x_{g+1} \in \mathbb{Z}_n^2$, 
and $n\in \mathbb{N}$, it holds that
\begin{align}\label{xx}
|\tilde{\mathcal{E}}_{\delta,x_1}[A_1]|\le Cn^{2t_1+\exp(g) \epsilon/2}, \quad
|\tilde{\mathcal{E}}_{\delta,x_{g+1}}[A_2]|\le Cn^{2t_2+\exp(h)\epsilon/2}.
\end{align} 
Therefore, if we let $\tilde{D}:=\{x\in \mathbb{Z}_n^2: d(x_1,x)\le Ln^{1-s+\delta}\}$, we have that for any $0<\delta<\epsilon$
\begin{align*}
&|\tilde{\mathcal{E}}_{\delta,x_1}[A]|  \\
\le&\sum_{ x_{g+1}\in \tilde{D} } 
|\{(x_2,\ldots,x_j):(x_2,\ldots,x_g)\in  \tilde{\mathcal{E}}_{\delta,x_1}[A_1],  
(x_{g+2},\ldots,x_j)\in \tilde{\mathcal{E}}_{\delta,x_{g+1}}[A_2]\}| \\
\le &Cn^{2-2s+2\delta} |\tilde{\mathcal{E}}_{\delta,x_1}[A_1]|\times| \tilde{\mathcal{E}}_{\delta,x_{g+1}}[A_2]| \\
\le& Cn^{2t_1+2t_2+2-2s+\exp(j) \epsilon/2}=Cn^{2t+\exp(j) \epsilon/2}. 
\end{align*}
The last inequality comes from $2 \delta +(\exp(g)+\exp(h))\epsilon/2
<(2+\exp(g)/2+\exp(h)/2)\epsilon<\exp(j)\epsilon/2$ for $j\ge 4$. 
Therefore, we  obtain the desired result. 
\end{proof}
\begin{rem}
The reason why we set ``$L$'' in (\ref{MM}) instead of  ``$2^j$'' is 
to ensure that we obtain (\ref{xx**}) and (\ref{xx}).  
\end{rem}
%%%%%%%%%%%%%%%%%%%%%%%%%%%%%%%%%%%%%%%%%%%%%%%%%%%%%%%%%%%%%%%%%%

\begin{proof}[Proof of Proposition \ref{kii}]
We  show the result by performing induction on $j \in \mathbb{N}$. 
It is obvious that the claim holds for $j=1$. 
Let us assume that the claim holds for $j-1$ and consider any $\vec{x}\in \mathcal{E}[ (n^{\eta }) , (n^\beta)]$ to show the claim for $j$. 
We set $1\le i_0, l_0 \le j$ such that $d(x_{i_0},x_{l_0})=\min_{1\le i\neq l\le j} d(x_i,x_l)$. 
Without loss of generality, we set $j=l_0$. 
%It is trivial the assumption yields the result for the general $(a_{i,l})_{i,l\in \Lambda}\in \mathcal{M}_j$. 
Then, as per the assumption, it is easy to extend the following: 
for $(x_1,\ldots,x_{j-1})\in \mathcal{E}[ (n^{\eta}) , (n^\beta) ](j-1,n)$ and $\delta>0$, 
there exists  
\begin{align}\label{ffa}
(a_{i,l})_{1\le i,l \le j-1}  \in  \mathcal{M}_{j-1}^{\beta,\eta}
\end{align}
such that for all sufficiently large $n\in \mathbb{N}$,
\begin{align*}
(x_1,\ldots,x_{j-1})\in \hat{\mathcal{E}}_{\delta/2}[(a_{i,l})_{1\le i,l \le j-1}  ](j-1,n).
\end{align*}
Set 
$\tilde{A}:= (\tilde{a}_{i,l})_{1\le i,l \le j}$ as follows: 
\begin{align*}
\begin{cases}
\tilde{a}_{i,l}&=\tilde{a}_{l,i}:=a_{i,l} \quad\text{ for any }1\le i, l \le j-1, \\
\tilde{a}_{j,l}&=\tilde{a}_{l,j}:=a_{i_0,l} \quad\text{ for any }1\le l \le j-1  \text{ with }l\neq i_0,\\
\tilde{a}_{j,i_0}&= \tilde{a}_{i_0,j}
:=(1-\frac{\log d(x_{i_0},x_{j})}{\log n}+\delta)\wedge (1-\eta),\\
\tilde{a}_{j,j}&:=1.
\end{cases}
\end{align*}
%It is obvious that $\tilde{a}_{j,i_0}\in [1-\beta, 1-\eta]$ for all sufficiently large $n\in \mathbb{N}$ 
%since $n^\eta \le d(x_{i_0},x_{j})\le n^\beta$ holds. 
We  prove that $\tilde{A}\in \mathcal{M}_j^{\beta,\eta}$ and $\vec{x}\in \hat{\mathcal{E}}_{\delta}[\tilde{A}]$. 
We first prove that $\tilde{A}\in \mathcal{M}_j^{\beta,\eta}$. 
It is obvious that the definition of $\tilde{A}$ yields that $\tilde{A}$ is symmetric 
and $1-\beta \le \tilde{a}_{i,l}\le 1-\eta$ for any $1\le i\neq l \le j$. 
Note that $\tilde{a}_{i_0,j}=\max_{1\le i, l\le j} \tilde{a}_{i,l}$ holds 
as $\hat{g}(s):=\max \{ b\in [1-\beta, 1-\eta]: 
2^{-j+1}n^{1-b} \le s \le 2^{j-1} n^{1-b+\delta/2}\}$ is monotonically decreasing; 
$\hat{g}(d(x_{i_0},x_{j}))\le  \tilde{a}_{i_0,j}$ for all sufficiently large $n\in \mathbb{N}$ with 
$2^{j-1}  \le n^{\delta/2}$ and 
$\hat{g}(d(x_i,x_l))\ge \tilde{a}_{i,l}$ hold for any $1\le i,l\le j$ with $i\neq i_0,j$. We only prove that for any $1\le i, l, p\le j$ with $ i\neq l,  l\neq p, p\neq i$,  
$(d)$  $\tilde{a}_{i,l}< \tilde{a}_{i,p}\Rightarrow \tilde{a}_{l,p}=\tilde{a}_{i,l}$ is as follows: 
\begin{enumerate}
 \item $i,l,p\neq j$ : (\ref{ffa}) yields $\tilde{A}\in \mathcal{M}_j^{\beta, \eta}$. 
Therefore, we obtain $(d)$. 
 \item $i=j$ and $p,l\neq i_0$ : If $\tilde{a}_{j,l}<\tilde{a}_{j,p}$, $a_{i_0,l}<a_{i_0,p}$ holds. 
 Therefore, (\ref{ffa}) yields $\tilde{a}_{l,p}=a_{l,p}=a_{i_0,l}=\tilde{a}_{j,l}$.  
 \item ($p=j$ and $i,l\neq i_0$)  or  ($l=j$ and $l,i \neq i_0$)  :  The proof is almost the same as above. 
%\item $i=l$, $i=p$ or $l=p$ : The assumption is contradict or the result is trivial. 
\end{enumerate}
Because $\tilde{A}$ is symmetric for $j$ and $i_0$, 
$(d)$ remains to be proven for the following cases: 
\begin{enumerate}
 \item $i=j$, $l=i_0$  :  The assumption is contradictory. 
 \item $i=j$, $p=i_0$  :  The result is trivial. 
 \item $l=j$, $p=i_0$  :   The assumption is contradictory. 
\end{enumerate}
Therefore, we obtain $\tilde{A}\in \mathcal{M}_j^{\beta, \eta}$. 
Finally, we prove $\vec{x}\in \hat{\mathcal{E}}_{\delta}[\tilde{A}]$.  
Note that the triangle inequality yields that 
for any $1\le l\le j-1$ with $l\neq i_0$, 
\begin{align*}
 d(x_{j},x_l) \le &d(x_{j},x_{i_0})+d(x_{i_0},x_l)\\
\le& 2^{j-1}n^{1-\tilde{a}_{i_0,j}+\delta}+2^{j-1}n^{1-a_{i_0,l}+\delta}
\le 2^j n^{1-\tilde{a}_{j,l}+\delta},
\end{align*}
and as $d(x_{j},x_l)+d(x_{j},x_{i_0})\ge d(x_{i_0},x_l) $ and $d(x_{j},x_l)\ge d(x_{j},x_{i_0}) $, 
\begin{align*} 
d(x_{j},x_l)
\ge \frac{1}{2}d(x_{i_0},x_l)
\ge  \frac{1}{2}\frac{1}{2^{j-1}}n^{1-\tilde{a}_{i_0,l}}
= \frac{1}{2^j}n^{1-\tilde{a}_{j,l} }.
\end{align*}
Therefore, as it is trivial that the corresponding result holds for $d(x_{j},x_{i_0})$, 
we have $\vec{x}\in \hat{\mathcal{E}}_{\delta}[\tilde{A}]$ 
and we obtain the desired result. 
\end{proof} 

%%%%%%%%%%%%%%%%%%%%%%%%%%%%%%%%%%%%%%%%%%%%%%%%%%%%%%%%%
Therefore, we obtain Proposition \ref{prop9*}. 
%%%%%%%%%%%%%%%%%%%%%%%%%%%%%%%%%%%%%%%%%%%%%%%%%%%%%%%%%%%%%%%%%%
Finally, we provide the following proposition,  
which is used in Section \ref{Theoh2}. 
\begin{prop}\label{prop4*}
For any $\epsilon>0$, there exists $\delta>0$ such that  
for any $A\in \mathcal{M}_j$ and $|\delta_{i,l}|\le \delta$ it holds that 
$(a_{i,l}+\delta_{i,l})_{i,l=1}^j$ is a regular matrix 
and 
$$\bigg|\chi(A)
- \chi((a_{i,l}+\delta_{i,l})_{i,l=1}^j)\bigg|\le \epsilon.$$ 
\end{prop}
\begin{rem}\label{prop4**}
It is trivial that Proposition \ref{prop4*} yields the following. 
For any $\epsilon>0$, there exists $\delta>0$ such that  
for any $n\in \mathbb{N}$, $A\in \mathcal{M}_j$, and $|\delta_{i,l}|\le \delta$, it holds that 
$(a_{i,l}+\delta_{i,l})_{i,l=1}^j$ is a regular matrix and 
$$\bigg|\frac{\chi(A)}{n}
- \chi((a_{i,l}+\delta_{i,l})_{i,l=1}^j)\bigg|\le \frac{\epsilon}{n}.$$ 
\end{rem}

\begin{proof}[Proof of Proposition \ref{prop4*}]
First, we prove that $\mathcal{M}_j$ is the closed set for each $j\in \mathbb{N}$. 
We consider $\max_{1\le i,l \le j}|a_{i,l}-a'_{i,l}|$ 
for $(a_{i,l})_{1\le i,l \le j}$, $(a'_{i,l})_{1\le i,l \le j} \in \mathcal{M}_j$ 
as the metric on $ \mathcal{M}_j$. 
Consider $(a^m_{i,l})_{1\le i,l \le j}\in \mathcal{M}_j$ and $(a^{\infty}_{i,l})_{1\le i,l \le j}$ such that  
$(a^m_{i,l})_{1\le i,l \le j}\to (a^{\infty}_{i,l})_{1\le i,l \le j}$ as $m\to \infty$. 
Note that it is trivial that $(a^{\infty}_{i,l})_{1\le i,l \le j}$ satisfies $(a)$ and $(b)$. 
Therefore, it suffices to show that $(a^{\infty}_{i,l})_{1\le i,l \le j}$ satisfies $(d)$. 
First, we assume that $a^{\infty}_{i,l}< a^{\infty}_{i,p}$ holds 
for some $1\le i, l, p\le j$ with $i\neq l, l\neq p, p\neq i$. 
For $k\in \mathbb{N}$, set $m_k:=\inf\{m>m_{k-1}: a^m_{i,l}>a^m_{i,p}\}$ with $m_1=1$ and $\inf \emptyset=0$. 
Then, $\sup m_k<\infty$, and as per the definition of $\mathcal{M}_j$, 
there exists $m_0\in \mathbb{N}$ such that for any $m\ge m_0$, 
$a^m_{l,p}= a^m_{i,l}$. 
Therefore, $a^{\infty}_{l,p}= a^{\infty}_{i,l}$ holds. 
$\mathcal{M}_j$ is the closed set and compact. 

%Next, we assume that $a^{\infty}_{i,l}= a^{\infty}_{i,p}$ holds for some $1\le i, l, p\le j$ with $i\neq l, l\neq p, p\neq i$. 
%If $\sup m_k<\infty$, by the definition of $\mathcal{M}_j$ there exists $m_0\in \mathbb{N}$ such that for any $m\ge m_0$, 
%$a^m_{l,p}\ge a^m_{i,l}$, and, hence, $a^{\infty}_{l,p}\ge a^{\infty}_{i,l}$. 
%If $\sup m_k=\infty$, by the definition of $\mathcal{M}_j$ we obtain  
%$a^{\infty}_{l,p}= a^{\infty}_{i,p}$, and, hence, 
%the assumption of $a^{\infty}_{i,l}= a^{\infty}_{i,p}$ yields $a^{\infty}_{l,p}= a^{\infty}_{i,l}$. 

Now, we prove the desired result. 
Note that 
\begin{align}\label{det0}
\max_{1\le i,l\le j} |a_{i,l}|\le1. 
\end{align}
In addition, Proposition \ref{prop2*} yields that det$A\neq 0$ holds for $A\in \mathcal{M}_j$. 
By the compactness of $\mathcal{M}_j$,   
we have $\inf_{A \in \mathcal{M}_j}|\text{det}A|\neq 0$. 
Therefore, (\ref{det0}) yields that 
there exists $\delta>0$ such that for any $\delta_{i,l}$ with $|\delta_{i,l}|\le \delta$, 
\begin{align}\label{det1}
\inf_{A\in \mathcal{M}_j}|\mathrm{det}(a_{i,l}+\delta_{i,l})_{1\le i,l \le j}|\neq 0.
\end{align} 
Thus, the first claim holds. 
Finally, it is trivial that (\ref{det0}) and (\ref{det1}) again yield the second claim, 
and therefore, we obtain the desired result. 
\end{proof}

%%%%%%%%%%%%%%%%%%%%%%%%%%%%%%%%%%%%%%%%%%%%%%%%%%%%%%%%%%%%%%%%%%
%%%%%%%%%%%%%%%%%%%%%%%%%%%%%%%%%%%%%%%%%%%%%%%%%%%%%%%%%%%%%%%%%%

%%%%%%%%%%%%%%%%%%%%%%%%%%%%%%%%%%%%%%%%%%%%%%%%%%%%%%%%%%%%%%%%%%%%%%%%%%%%%%%%%%%%%%%%%%%%%%%%%%
\section{Appendix }
\subsection{Computation of exponents}\label{appe}
In this section, we provide the estimation for the monotonicity of the exponents. 
\begin{lem}\label{prop7*}
For any $j\ge2$ and $0<\alpha, \beta<1$
\begin{align*}
\text{ }\hat{\rho}_j (\alpha,\beta)-\hat{\rho}_{j-1}(\alpha,\beta)\ge 0.
%&(B)\text{ }\rho_j (\alpha,\beta)-\rho_{j-1}(\alpha,\beta)\ge 0.
\end{align*}
\end{lem}
\begin{rem}
As discussed previously, this result yields Theorem \ref{h2}. 
In addition, the above lemma is equivalent to Theorem \ref{h2}. 
%and (A) is equivalent to Theorem \ref{h4}. 
For any $0< \alpha,\beta <1$,
\begin{align*}
|\{ \vec{x}\in \mathcal{L}_n(\alpha)^j:
d(x_i,x_l)\le n^{\beta} \text{ for any }1\le i,l \le j  \}|
\end{align*}
is monotonically increasing in $j\in \mathbb{N}$. 
Therefore, Theorem \ref{h2} naturally yields the above lemma. 
\end{rem}
\begin{rem}
We omit the proof because we only need long and elementary computations. 
\end{rem}

\section*{Acknowledgments. }
We are grateful to Prof. Kazumasa Kuwada and Kohei Uchiyama for their interesting discussions. 
In addition, we would like to thank Yoshihiro Abe for his helpful comments.

%%%%%%%%%%%%%%%%%%%%%%%%%%%%%%%%%%%%%%%%%%%%%%%%%%%%%%%%%%%%%%%%%%
%%%%%%%%%%%%%%%%%%%%%%%%%%%%%%%%%%%%%%%%%%%%%%%%%%%%%%%%%%%%%%%%%%

\end{document}